\documentclass[12pt]{amsart}

\usepackage{mathtools}
\usepackage{titlesec,caption, subcaption}
\usepackage{tikz-cd}
\usepackage{amsmath,amsthm,amsfonts,amscd,amssymb,amsopn,enumerate,enumitem}
\usepackage{mathrsfs,xcolor,hyperref,stmaryrd}
\usepackage{wasysym}
\usepackage{cleveref}
\usepackage[all]{xy}
\usepackage{fullpage}
\usepackage{verbatim}
\usepackage{colonequals}
\usepackage[new]{old-arrows}
\usepackage[OT2,T1]{fontenc}

\titleformat{name=\section}{}{\thetitle.}{0.8em}{\centering\scshape}
\titleformat{name=\subsection}[runin]{}{\thetitle.}{0.8em}{\bfseries}[.]
\titleformat{name=\subsubsection}[runin]{}{\thetitle.}{0.8em}{\bfseries}[.]

\DeclareSymbolFont{cyrletters}{OT2}{wncyr}{m}{n}
\DeclareMathSymbol{\Sha}{\mathalpha}{cyrletters}{"58}

\DeclareMathOperator{\ssd}{\textnormal{ssd}}
\DeclareMathOperator{\gssd}{\textnormal{gssd}}
\DeclareMathOperator{\cd}{\textnormal{cd}}

\DeclareMathOperator{\subalg}{\textnormal{-subalgebras of }}
\DeclareMathOperator{\per}{per}
\DeclareMathOperator{\coker}{coker}
\DeclareMathOperator{\Spec}{Spec}
\DeclareMathOperator{\ind}{ind}
\DeclareMathOperator{\Frac}{Frac}
\DeclareMathOperator{\vect}{Vect}
\DeclareMathOperator{\GL}{GL}
\DeclareMathOperator{\tors}{\mathbf{-torsor}}
\DeclareMathOperator{\Mor}{Mor}
\DeclareMathOperator{\Aut}{Aut}
\DeclareMathOperator{\Uaut}{\underline{Aut}}
\DeclareMathOperator{\Galalg}{\mathbf{-Alg}_{Gal}}
\DeclareMathOperator{\red}{red}

\theoremstyle{plain}
\newtheorem{thm}{Theorem}[section]
\newtheorem{theorem}[thm]{Theorem}
\newtheorem*{theorem*}{Theorem}
\newtheorem{lemma}[thm]{Lemma}
\newtheorem*{lemma*}{Lemma}
\newtheorem{corollary}[thm]{Corollary}
\newtheorem{proposition}[thm]{Proposition}
\newtheorem*{proposition*}{Proposition}

\newtheorem*{conjecture*}{Conjecture}

\newtheorem*{question*}{Question}

\newtheorem*{fact*}{Fact}

\theoremstyle{definition}
\newtheorem{definition}[thm]{Definition}

\newtheorem*{conventions*}{Conventions}

\newtheorem*{acknowledgements*}{Acknowledgements}

\newtheorem{remark}[thm]{Remark}

\newcommand{\Z}{\mathbb{Z}}
\newcommand{\Q}{\mathbb{Q}}

\newcommand{\id}{id}

\newcommand{\coho}[1]{H^{i}(#1, \mu_l^{\otimes i - 1})}
\newcommand{\cohom}[2]{H^{i}(#1, \mu_{#2}^{\otimes i - 1})}
\newcommand{\nrcoho}[1]{H^i_{\textnormal{nr}}(#1, \mu_l^{\otimes i - 1})^{\mathcal{X}}}

\usepackage{xcolor}
\definecolor{c1}{RGB}{203, 75, 48}

\hypersetup{
    colorlinks = true,
}

\title{Patching for \'etale algebras and the period-index problem for higher degree Galois cohomology groups over Hensel semi-global fields}

\subjclass[2020]{11E72, 12G05, 12E30 (primary); 14B12, 14H25 (secondary)}
\keywords{Artin approximation, Brauer group, \'etale algebras, Galois cohomology, henselian rings, local-global principle, patching, period-index problem, semi-global fields}
\date{\today}

\begin{document}

\author{Yidi Wang}

\begin{abstract} In this manuscript, we present a partial generalization of the field patching technique initially proposed by Harbater-Hartmann to Hensel semi-global fields, i.e., function fields of curves over excellent henselian discretely valued fields. More specifically, we show that patching holds for \'etale algebras over such fields and a suitable set of overfields. Within this new framework, we further establish a local-global principle for higher degree Galois cohomology groups over Hensel semi-global fields. As an application, we extend a recent result regarding a uniform  period-index bound for higher degree Galois cohomology classes by Harbater-Hartmann-Krashen to Hensel semi-global fields. Additionally, we prove such a bound for coefficient groups of non-prime orders.
\end{abstract}

\maketitle

\section{Introduction}

Patching, which is inspired by the idea of ``cut and paste'' from topology and analysis, has been an effective tool in solving problems in Galois theory. The \emph{field patching} technique introduced by Harbater-Hartmann (\cite{HH10}) employs the idea that giving an algebraic structure over certain fields is equivalent to giving the structure over a suitable collection of overfields.
In \cite[Proposition 2.1]{HH10}, the authors gave criteria for patching to hold for vector spaces and showed that these hold for semi-global fields, i.e., function fields of curves over complete discretely valued fields. From this, one may deduce patching for other algebraic objects, e.g., commutative algebras, torsors, etc., by treating them as vector spaces with additional structure (see \cite[Theorem 7.1]{HH10} and \cite[Theorem 2.3]{HHK15}). Field patching simplifies the machinery in older forms of patching, e.g., formal and rigid patching, etc., and has found applications to various areas in arithmetic geometry such as local-global principles and the inverse problem in differential Galois theory in the past two decades. This raises the question of whether this technique can be extended to more general fields. 

In this manuscript, we study the field patching technique over \emph{Hensel semi-global fields}, i.e., function fields of curves over excellent henselian discretely valued fields. Though techniques that work for objects defined over complete rings can often be directly applied to objects defined over henselian rings, several difficulties present in our setting. 

First, we do not know what to expect for patching for vector spaces here. It is essentially more difficult to apply the criteria in \cite[Proposition 2.1]{HH10} to study whether patching holds for vector spaces over a Hensel semi-global field,
as the structure of henselian rings is very different. In a different perspective, field patching over semi-global fields proves many analogues of the results in formal patching for coherent sheaves of formal schemes. The proof of formal patching relies on Grothendieck's Existence Theorem, or Formal GAGA (GFGA).
However, Devadas showed that Henselian GAGA (GHGA) holds for finitely presented subsheaves of algebraizable sheaves on finitely presented and proper schemes over henselian rings of positive characteristic \cite[Theorem 4.3.1, Corollary 4.3.4]{Devadas23} but fails when the base ring is of equicharacteristic zero \cite[Example 4.4.1]{Devadas23} based on an example by de Jong (\cite{deJong19}); whether GHGA holds over mixed characteristic henselian rings remains an open question. Hence, in general, one cannot deduce patching for \'etale (Galois) covers as a consequence.

Despite the restrictions on the characteristic of the base ring in patching for coherent sheaves on henselian schemes, 
we prove that patching holds for torsors under finite groups over Hensel semi-global fields and a suitable set of overfields, called \emph{Hensel patches}, defined from the underlying geometry of a model of~$F$ \emph{without} any assumption on the characteristic of $F$. See ~\Cref{sec:patching} for the precise definitions.

\begin{theorem}[\Cref{patching_torsors}]\label{theorem1}
    Let~$F$ be a Hensel semi-global field,~$F_1, F_2$ be Hensel patches, and~$F_0$ be the Hensel branch field. Let $G$ be a finite group. The map
    \begin{equation*}
        \beta \colon G\tors(F) \to G\tors(F_1)\times_{G\tors(F_0)} G\tors(F_2)
    \end{equation*} given by base change is an equivalence of categories. In particular, when $G = S_n$, $\beta$ gives an equivalence of categories of \'etale algebras.
\end{theorem}

\noindent We prove \Cref{theorem1} by descending the patching result for \'etale algebras over semi-global fields to Hensel semi-global fields, using the Artin Approximation theorem \cite{Artin1969}. In fact, over Hensel semi-global fields, whether patching holds for vectors spaces remains unknown, even though patching for \'etale algebras holds.

We then apply patching for \'etale algebras to prove a period-index bound for higher degree Galois cohomology groups. The classical period-index problem relates the \emph{period} of a Brauer class, i.e., the order of a class in the Brauer group, to its \emph{index}, i.e., the degree of the minimal splitting field of a central simple algebra representing this class.
It is known that the index of a class always divides a power of its period. Moreover, for various classes of fields, there exists an integer~$d$ such that the index divides a~$d$-th power of the period for all classes. For example,~$d = 1$ when~$F$ is a global field;~$d = 2$ when~$F$ is a one-variable function field over~$\Q_p$ (see \cite{Saltman97}). Bounds concerning semi-global fields are also available in \cite{HHK09} and \cite{Lieblich11}.

In terms of Galois cohomology, the~$m$-torsion part of the Brauer group, i.e., the set of equivalence classes of central simple~$F$-algebras with period dividing~$m$, is identified with~$H^2(F, \mu_m)$. Following Kato (\cite{Kato1986}), one can then view~$\cohom{F}{m}$ as a generalization of the~$m$-torsion subgroup of the Brauer group and ask for a bound for splitting these higher cohomology classes. When~$i > 2$, this question becomes more mysterious as we lose most of the geometric interpretations like Severi-Brauer varieties, gerbes, etc., attached to the Brauer group. Motivated by Gosavi's work on the ``simultaneous index'' (\cite{Gosavi22}), Harbater-Hartmann-Krashen defined the \emph{generalized stable splitting dimension} as the minimal degree of a finite extension of~$E$ that simultaneously splits a finite collection of~$n$ classes in~$\coho{E}$  for all finite field extensions~$E/F$. Such a field invariant is denoted by~$\gssd_l^i(F)$ (by~$\ssd_l^i(F)$ if~$n = 1$). They used patching to bound~$\gssd_l^i(F)$ for a semi-global field~$F$ (\cite{HHK23}), when~$l$ is prime. Given the generalized patching result, we are able to relax their assumption on the field. In \Cref{sec_non_prime}, we also prove such a bound for~$\cohom{F}{m}$, where~$m$ is not necessarily a prime number.

\begin{theorem}[\Cref{main_theorem}, \Cref{general_theorem}]\label{theorem2}
    Let~$K$ be an excellent henselian discretely valued field with residue field~$k$. Let~$F$ be a one-dimensional function field over~$K$. Let~$l$ be a prime that is not the characteristic of~$k.$ Let~$k(x)$ denote the rational function field in one variable over~$k$. Then
    \begin{equation*}
        \gssd_l^i(F) \leq \gssd_l^i(k) + \gssd_l^i(k(x)) + 
        \begin{cases}
            2 \textnormal{ if } l \neq 2\\
            3 \textnormal{ if } l = 2
        \end{cases}
    \end{equation*}
    Let~$m = l_1^{e_1}\cdots l_k^{e_k}$, where each~$l_s$ is prime, pairwisely distinct and~$\operatorname{char}(F) \nmid m$. Then 
    \begin{equation*}
        \ssd_m^i(F) \leq \max_{1 \leq s \leq k} \{\ssd_{l_s}^i(F)\}.
    \end{equation*}
\end{theorem} \noindent In fact, the second part of the theorem holds over arbitrary fields. The approach for proving \Cref{theorem2} is to construct ``local splitting fields''  and then use the patching result for \'etale algebras to obtain a splitting field over~$F$. To check that a suitable subfield of the \'etale algebra constructed by patching is indeed a splitting field, we need a local-global principle for higher degree cohomology groups for which we prove in ~\Cref{sec_lgp}.

 The manuscript is organized as follows. \Cref{sec:artin} recalls henselian rings and the Artin Approximation theorem. In \Cref{sec:patching}, we introduce the field patching technique over Hensel semi-global fields. \Cref{sec_lgp} provides a local-global principle for higher degree cohomology groups. We then apply this principle to the generalized period-index bound in \Cref{sec:period-index}. In \Cref{sec_non_prime}, we extend the bound to coefficient groups of non-prime orders. 

\begin{acknowledgements*} The author was partially supported by the NSF grant DMS-2102987. This manuscript contains results from the author’s doctoral dissertation at the University of Pennsylvania. The author would like to thank David Harbater and Daniel Krashen for enlightening discussions and helpful comments, and Julia Hartmann for her patient advising and encouragement.
\end{acknowledgements*}

\section{Henselian rings and the Artin Approximation Theorem}\label{sec:artin}
The Artin Approximation theorem plays an important role in drawing connections between algebraic objects over an excellent henselian discrete valuation ring and its completion. In this section, we first recall the properties of henselian rings and  the Artin Approximation theorem (\cite{Artin1969}). Then we introduce a modified version of it to make it easier to apply to Hensel semi-global fields. Readers who are familiar with henselian rings and Artin Approximation may skip this section.

The classical study of henselian rings and henselization works primarily with local rings, but in this manuscript, sometimes we need a general treatment that works for any commutative rings, namely \emph{henselian pairs}. We mostly follow \cite{Elkik} and \cite{stacks-project}. A \emph{pair}~$(A, I)$ consists of a commutative ring~$A$ and an ideal~$I$. An \emph{\'etale neighborhood} of the pair~$(A, I)$ consists of a pair~$(A', I')$ and an \'etale ring map~$\alpha \colon A \to A'$ such that~$IA' = I'$ and~$A/I \cong A'/I'$. Geometrically, an \'etale neighborhood of~$(\Spec(A), V(I))$ is an affine \'etale map~$\Spec(A') \to \Spec(A)$ which is the identity over the closed subscheme~$V(I)$. A pair~$(A, I)$ is \emph{henselian} if for any \'etale neighborhood~$(A', I')$, there is an~$A$-algebra homomorphism~$A'\to A$ lifting the \'etale map~$\alpha$. Note that this definition coincides with the most natural definition of henselian ring, i.e., a ring that satisfies Hensel's Lemma. See \cite[Lemma 15.11.6]{stacks-project}. The \emph{henselization of a pair~$(A, I)$} is a ring extension~$\phi \colon (A, I) \to (A^h, IA^h)$, where~$(A^h, IA^h)$ is a henselian pair having the following universal property: for any henselian pair~$(B, J)$ with morphism~$\psi \colon (A, I)\to (B, J)$, there exists a unique morphism~$v \colon (A^h, IA^h) \to (B, J)$ such that~$\psi = v \circ \phi$. The henselization~$A^h$ of~$A$ can be constructed by taking the direct limit of all \'etale neighborhoods (see proof of \cite[Lemma 15.2.1]{stacks-project}). It is also called the henselization of~$A$ at~$I$.

\begin{theorem}\cite[Theorem 1.10]{Artin1969}\label{artin1}
    Let~$R$ be a field or an excellent discrete valuation ring, and and let~$A$ be the henselization of an~$R$-algebra of finite type at a prime ideal. Let~$\mathfrak{m}$ be a proper ideal of~$A$. Given an a system of polynomial equations
    \begin{equation*}
        f(Y_1, \cdots, Y_n) = 0
    \end{equation*} with coefficients in~$A$, a solution~$\bar{y} = (\bar{y}_1, \cdots, \bar{y}_n)$ in the~$\mathfrak{m}$-adic completion~$\widehat{A}$ of~$A$, and an integer~$c$, there exists a solution~$y = (y_1, \cdots, y_n) \in A$ with~$\bar{y}_i \equiv y_i \mod \mathfrak{m}^c$.
\end{theorem}

There is an alternative statement of Artin's Approximation Theorem, which is usually easier to apply in categorical settings.

\begin{theorem}\cite[Theorem 1.12]{Artin1969}\label{artin2}
     Let notation be as in \Cref{artin1}, and let
    \begin{equation*}
        F \colon \left(A\textnormal{-algebras}\right) \longrightarrow \left(\textnormal{Sets}\right),
    \end{equation*} be a functor sending filtered colimits to filtered colimits (or locally of finite presentation). Then given~$\hat{\xi} \in F(\widehat{A})$ and an integer $c$, there exists~$\xi \in F(A)$ such that~$\hat{\xi} \equiv \xi \mod \mathfrak{m}^c.$
\end{theorem}

We may simplify the hypothesis of the Artin Approximation Theorem to make it easier to apply to function fields. Let~$T$ be an excellent henselian discrete valuation ring and~$R$ be a~$T$-algebra of finite type. Let~$R^h$ denote the henselization of~$R$ at a prime ideal. Let~$\widehat{R}$ denote the completion of~$R^h$ at a proper ideal~$I$ of~$R^h.$ Suppose further that~$\widehat{R}$ is an integral domain. We may write~$\widehat{R}$ as a filtered direct system of~$R^h$-subalgebras~$\widehat{R} = \varinjlim_i A_i$, where each~$A_i$ is an~$R^h$-algebra of finite type. Moreover,~$A_i$ is an integral domain for all~$i$ since~$\widehat{R}$ is an integral domain. Note that since each~$A_i$ is a~$R^h$-subalgebra of~$\widehat{R}$, each map~$\varphi_{ij} \colon A_i \to A_j$ is an inclusion for~$i \leq j$.

\begin{proposition}\label{artin_modified}
    Let notation be as above. Let 
    \begin{equation*}
        F \colon (R^h\subalg \widehat{R}) \longrightarrow (\textnormal{Sets}),
    \end{equation*} be a functor such that~$F(\widehat{R}) = F(\varinjlim_i A_i) = \varinjlim_i(F(A_i))$. Then given~$\hat{\alpha} \in F(\widehat{R})$ and an integer $c$, there exists~$\alpha \in F(R^h)$ such that~$\alpha \equiv \hat{\alpha} \mod I^c$.
\end{proposition}

\begin{proof}
    The proof is similar to that of \cite[Theorem 1.12]{Artin1969}. For~$\hat{\alpha} \in F(\widehat{R})$, there exists~$i$ such that~$\hat{\alpha} \in F(A_i)$. Since~$A_i$ is a~$R^h$-algebra of finite type, we may write~$A_i = R^h[Y_1, \dots, Y_m]/(f_1, \dots, f_n)$, and there is an injective~$R^h$-algebra homomorphism~$A_i \hookrightarrow~\widehat{R}$ given by~$Y_j \mapsto y_j$,~$j = 1, \dots, m$, where~$(y_1, \dots, y_m)$ is a solution of the system of polynomial equations~$f_1, \dots, f_n$  induced by ~$\hat{\alpha}$ in~$\widehat{R}^m$. By \Cref{artin1}, for any integer $c$, there exists a solution~$(y_1', \dots, y_m')$ in~$(R^h)^m$ of the system of polynomial equations~$f_1, \dots, f_n$ such that $y_i \equiv y_i' \mod I^c$. This solution induces an injective homomorphism~$A_i \hookrightarrow R^h \colon Y_j \mapsto y_j'$ and this homomorphism further induces an element~$\alpha \in F(R^h)$ that satisfies $\alpha \equiv \hat{\alpha} \mod I^c$.
\end{proof}

There is also a version of generalized Artin Approximation Theorem that applies to Henselian pairs.

\begin{theorem}\cite[Theorem 1.1]{DH90}\label{theorem:artin_henselian_pair}
    Let~$(A, I)$ be a pair, where~$A$ is an excellent normal domain of dimension at most two and~$I$ is an ideal of~$A$. Let~$A^h$ and~$\widehat{A}$ denote the henselization and completion of the pair~$(A, I)$ (as in the beginning of \Cref{sec:artin}). Given an a system of polynomial equations
    \begin{equation*}
        f(Y_1, \cdots, Y_n) = 0
    \end{equation*} with coefficients in~$A^h$, a solution~$\bar{y} = (\bar{y}_1, \cdots, \bar{y}_n)$ in~$\widehat{A}$, and a non-negative integer~$c$, there exists a solution~$y = (y_1, \cdots, y_n) \in A$ with~$\bar{y}_i \equiv y_i \mod I^c$.
\end{theorem}

\begin{lemma}\label{linearly_disjoint}
    Let~$T$ be an excellent henselian discrete valuation ring. Let~$A$ be a~$T$-subalgebra of~$\widehat{T}$ of finite type that is an integral domain. Let~$K$ be the fraction field of~$T$. Let~$F$ be a field extension of~$K$ such that~$K$ is relatively algebraically closed in~$F$. Then~$F \otimes_{T} A$ is an integral domain, whence the total ring of fractions~$\operatorname{Frac}(F \otimes_{T} A)$ is a field.
\end{lemma}

\begin{proof}
    Let~$A' \coloneqq K \otimes_{T} A$, where~$K = \operatorname{Frac}(T)$. Then~$A'$ is a~$K$-subalgebra of~$\widehat{K} = \Frac(\widehat{T})$ of finite type. We must show that~$F \otimes_{T} A = F \otimes_K A'$ is an integral domain. Since~$F/K$ is flat,~$F \otimes_K A' \hookrightarrow F \otimes_K \operatorname{Frac}(A')$. Hence, it suffices to show that~$L \coloneqq F \otimes_K \operatorname{Frac}(A')$ is an integral domain. 

    First, by \cite[Theorem 78]{matsumura80},~$T$ is a Nagata ring. Moreover, the local Nagata domain~$T$ has the property that~$\Frac(\widehat{T})$ is separable over~$\Frac(T)$ by \cite[Theorem 71]{matsumura80}. Therefore, there exists a transcendence basis~$x_1, \dots, x_n$ over~$K$ such that~$$\Frac(A') = K(x_1, \dots, x_n)(\alpha_1, \dots, \alpha_m),$$ where each~$\alpha_i$ is separable and algebraic over~$K(x_1, \dots, x_n)$. For simplicity, write ~$\operatorname{Frac}(A') = K(\underline{x})(\underline{\alpha}).$ Therefore,
    \begin{equation*}
        L = F \otimes_K K(\underline{x})(\underline{\alpha}) = F \otimes_K K(\underline{x}) \otimes_{K(\underline{x})} K(\underline{x})(\underline{\alpha}) = F(\underline{x}) \otimes_{K(\underline{x})} K(\underline{x})(\underline{\alpha}).
    \end{equation*} Since~$K$ is relatively algebraically closed in~$F$,~$K(\underline{x})$ is relatively algebraically closed in~$F(\underline{x})$. By \cite[Lemma 24.13]{fvk},~$F(\underline{x})$ and~$K(\underline{x})(\underline{\alpha})$ are linearly disjoint, whence~$L$ does not have zero divisors. 
\end{proof}

\section{Patching over Hensel Semi-global fields}\label{sec:patching}

The field patching technique was invented by Harbater-Hartmann in \cite{HH10}, which later found more applications in arithmetic geometry like studying the period-index problem for Brauer groups and~$u$-invariants for quadratic forms. The idea is to show that giving an algebraic structure over a field is equivalent to giving the structure over a suitable collection of overfields. We elaborate below for the formal definition.

Let~$E \subseteq E_1, E_2 \subseteq E_0$ be fields. Let~$\mathcal{C}$ denote the category of certain algebraic objects over~$E$, e.g., vector spaces, torsors, associative algebras, etc. Let~$\mathcal{C}_i$ denote the category of algebraic objects of the same type as~$\mathcal{C}$ over~$E_i$. Let
\begin{equation*}
    \Phi \colon \mathcal{C} \longrightarrow \mathcal{C}_1 \times_{\mathcal{C}_0} \mathcal{C}_2
\end{equation*} be the natural map given by base change. Objects in~$\mathcal{C}_1 \times_{\mathcal{C}_0} \mathcal{C}_2$ are called \emph{patching problems}. These are triples~$(M_1, M_2; \phi)$, where each~$M_i \in \mathcal{C}_i$ and~$\phi \colon M_{1,E_0} \to M_{2,E_0}$ is an isomorphism over~$E_0$.  A \emph{morphism}~$\beta$ of patching problems~$(M_1, M_2;\phi)$ and~$(M_1', M_2';\phi')$ is a pair~$(\beta_1, \beta_2)$, where~$\beta_i \colon M_i \to M_i'$ satisfies~$\phi'\circ \beta_1 = \beta_2 \circ \phi$ for~$i = 1, 2$. We say \emph{patching holds over~$(E, E_1, E_2, E_0)$}  if the functor~$\Phi$ is an equivalence of categories. This implies that every patching problem~$(M_1, M_2;\phi)$ has a unique \emph{solution}, i.e., an algebraic object~$M \in \mathcal{C}$ that induces~$M_1, M_2$ and the isomorphism~$\phi$. 

In this section, we first recall the setup in \cite{HH10} for patching over \emph{semi-global fields}, and then introduce patching over \emph{Hensel semi-global fields}.

\subsection{Patching over semi-global fields}
Let~$K$ be a complete discretely valued field with valuation ring~$T$, uniformizer~$t$, and residue field~$k$. Let~$F$ be a one-variable function field over~$K$, in which~$K$ is algebraically closed. Such a field~$F$ is called a \emph{semi-global field}. A \emph{normal model}~$\mathcal{X}$ of~$F$ is an integral~$T$-scheme with function field~$F$ that is flat, projective and of relative dimension one over~$T$, and that is normal as a scheme. We say~$\mathcal{X}$ is \emph{regular} if it is regular as a scheme.
A regular model exists by the main theorem of \cite{Lipman78}. 

Let~$P$ be a point on the closed fibre~$X$ of~$\mathcal{X}$. Let~$\widehat{R}_P \coloneqq \widehat{\mathcal{O}}_{\mathcal{X}, P}$ be the complete local ring of~$\mathcal{X}$ at~$P$; and let~$F_P$ be the fraction field of~$\widehat{R}_P.$

Let~$U$ be a non-empty connected affine open subset of~$X$. Let~$R_U$ be the subring of~$F$ consisting of rational functions that are regular at each point of~$U.$ Let~$\widehat{R}_U$ be the~$t$-adic completion of~$R_U$. Let~$F_U$ be the fraction field of~$\widehat{R}_U.$

For~$P$ in the closure~$\bar{U}$ of~$U$, a \emph{branch} of~$X$ at~$P$ is a height-one prime ideal~$\wp$ of~$\widehat{R}_P$ that contains the uniformizer~$t$.  Let~$R_{\wp}$ be localization of~$\widehat{R}_P$ at~$\wp.$ Note that~$R_{\wp}$ is a discrete valuation ring. Now we take the~$t$-adic completion of~$R_{\wp}$, denoted by~$\widehat{R}_{\wp}$. Then~$\widehat{R}_U, \widehat{R}_P \subseteq \widehat{R}_{\wp}$ (see \cite[Section 6]{HH10}). Denote the fraction field of~$\widehat{R}_{\wp}$ by~$F_{\wp}$. It follows that~$F_P, F_U \subseteq F_{\wp}.$ The fields~$F_P$ and~$F_U$ are called \emph{patches}. The field~$F_{\wp}$ is called a \emph{branch field}.

Let~$E \subseteq E_1, E_2 \subseteq E_0$ be fields. Let~$\vect{(E_i)}$ denote the category of finite dimensional vector spaces over~$E_i$. Let
\begin{equation*}
    \Phi \colon \vect{(E)} \longrightarrow \vect{(E_1)} \times_{\vect{(E_0)}} \vect{(E_2)}
\end{equation*} be the natural map given by base change.  Objects in~$\vect{(E_1)} \times_{\vect{(E_0)}} \vect{(E_2)}$ are called \emph{patching problems}. These contain triples~$(V_1, V_2; \phi)$, where~$V_i$ is a finite dimensional~$E_i$-vector space and~$\phi \colon V_1 \otimes_{E_1} E_0 \to V_2 \otimes_{E_2} E_0$ is an isomorphism of~$E_0$-vector spaces. If~$\Phi$ is an equivalence of categories, then every patching problem~$(V_1, V_2;\phi)$ has a unique \emph{solution}, i.e., a~$E$-vector space~$V$ that induces~$V_1, V_2$ and~$\phi$. The following conditions for patching for vector spaces to hold were given in \cite[Proposition 2.1]{HH10}:
\begin{itemize}
    \item~$E = E_1 \cap E_2$.
    \item for any~$A_0 \in \GL_n(E_0)$, there exists~$A_i \in \GL_n(E_i)$ such that~$A_0 = A_1A_2$.
\end{itemize}

This was shown to hold over semi-global fields and patches.

\begin{theorem}(\cite[Theorem 6.4, Corollary 6.7]{HH10}) Let notation be as above.
    The following functor given by base change 
    \begin{equation*}
        \vect(F) \to \vect(F_P) \times_{\vect(F_{\wp})} \vect(F_U)
    \end{equation*}is an equivalence of categories.
\end{theorem}

\noindent This theorem also extends to other objects like commutative algebras, torsors, etc., by treating them as vector spaces with additional structures (see \cite[Theorem 7.1]{HH10} and \cite[Theorem 2.3]{HHK15}).

\subsection{Notation: Hensel semi-global fields and Hensel patches}\label{subsec:patching}

In this section, we introduce the notion of Hensel semi-global fields and prove a patching result for torsors under finite constant groups over such fields.

Let~$T$ be an excellent henselian discrete valuation ring with a uniformizer~$t$ (not necessarily complete). Let~$K$ be the fraction field of~$T$ and~$k$ be the residue field of~$T$. Such a field~$F$ is called a \emph{Hensel semi-global field}. Let~$F$ be a one-variable function field over~$K$ such that~$K$ is algebraically closed in~$F$. Let~$\widehat{T}$ denote the completion of~$T$ at the maximal ideal. Note that~$T$ and~$\widehat{T}$ have the same residue field~$k$. Let~$\widehat{K}$ denote the fraction field of~$\widehat{T}$. Then~$\widehat{K}$ is the completion of~$K$. Let~$\mathcal{X}$ be a normal model for~$F$. As discussed in Section~3.2 of \cite{HHKCPS19}, since~$T$ is henselian,~$\widehat{\mathcal{X}} \coloneqq \mathcal{X} \times_T \widehat{T}$ is also normal and has the same closed fibre~$X$ as~$\mathcal{X}$. Moreover, the function field~$\widehat{F}$ of~$\widehat{\mathcal{X}}$ is~$\textnormal{Frac}(F \otimes_K \widehat{K}).$ 

Let~$\mathcal{P} \subset X$ be a non-empty subset containing all the points at which distinct irreducible components of~$X$ meet. Thus the open complement~$X \backslash \mathcal{P}$ is a disjoint union of finitely many irreducible affine~$k$-curves~$U$. Let~$\mathcal{U}$ denote the collection of these open sets~$U$.

Let~$P \in \mathcal{P}$ and let~$R_P$ be the local ring of~$\mathcal{X}$ at~$P$. Let~$\mathfrak{m}_P$ be the maximal ideal of~$R_P$. Take the henselization of the pair~$(R_P, \mathfrak{m}_P)$, we obtain a henselian pair~$(R_P^h, \mathfrak{m}_PR_P^h)$. Let~$F_P^h$ be the fraction field of~$R_P^h.$ Similarly, let~$U\in \mathcal{U}$ and let~$R_U$ be the subring of~$F$ that consists of the rational functions that are regular along~$U.$ Let~$I$ be the ideal that defines~$U$ in~$\operatorname{Spec}(R_U).$ By \cite[Remark 3.1.1]{Harbater03},~$R_U$ is a Zariski ring. Then taking the henselization of~$(R_U, I)$, we obtain a henselian pair~$(R_U^h, IR_U^h)$ (see Section~\ref{sec:artin} for details). The fraction field of ~$R_U^h$ is denoted by~$F_U^h$. We call the fields~$F_P^h$ and~$F_U^h$ the \emph{Hensel patches}.

If~$P \in \bar{U}$, a \emph{branch} of~$X$ at~$P$ is a height-one prime ideal~$\wp$ of~$R_P^h$ that contains the uniformizer~$t$.  Let~$R_{\wp}$ be the localization of~$R_P^h$ at~$\wp.$ Note that~$R_{\wp}$ is a discrete valuation ring. Now we take the henselization of~$(R_{\wp}, \wp R_{\wp})$ and obtain a henselian ring~$R_{\wp}^h$. By construction,~$R_{P}^h \subseteq R_{\wp}^h$. To see that $R_U^h \subseteq R_{\wp}^h$, we note that the localization of $R_P$ and $R_U$ at the generic point $\eta$ of $U$ are the same, so by construction $(R_U)_{\eta} = (R_P)_{\eta} \subseteq R_{\wp}^h$. Therefore, $R_U \subseteq R_{\wp}^h$, so is $R_{U}^h$.
Denote the fraction field of~$R_{\wp}^h$ by~$F_{\wp}^h$. We call such a field a \emph{Hensel branch field}. It follows that~$F_P^h, F_U^h \subseteq F_{\wp}^h.$ Moreover, note that~$F_P^h$ is dense in~$F_{\wp}^h$. Indeed,~$F_P^h$ is contained in the fraction field of the completion of~$R_{\wp}$ at~$\wp$ and also has the~$\wp$-adic topology.

Now we make a few remarks about the relationship between semi-global patches and Hensel patches. The complete local rings of~$\mathcal{X}$ and~$\widehat{\mathcal{X}}$ at~$P \in X$ are isomorphic, so we may write~$\widehat{R}_P$ and thus~$F_P$ unambiguously for the patches. Similarly, for an affine open subset~$U$ of~$X$, we may also write~$\widehat{R}_U$ and~$F_U$ without ambiguity. Moreover, taking the completion of~$R_P^h$,~$R_U^h$ with respective to the ideals~$\mathfrak{m}_P$,~$IR_U^h$ recovers~$\widehat{R}_P$ and~$\widehat{R}_U$ respectively.

Though patching over semi-global fields and patches was first proven for vector spaces and then for other algebraic objects treated as vector spaces with additional structures, it is essentially hard to carry over the strategy over Hensel semi-global fields. Indeed, the proof for patching for vector spaces in \cite{HH10}, in particular for the factorization property,  adopts the argument of inductively quotienting by a power of the uniformizer~$t$ of the base ring~$\widehat{T}$ and then solving the entries over the patches. This argument relies heavily on the fact that the base ring~$\widehat{T}$ is complete and does not work over henselian rings. In the following section, we will use the Artin Approximation Theorem (\cite[Theorem 1.10]{Artin1969} and Proposition~\ref{artin_modified}) to descend the patching result for torsors under finite constant groups over semi-global fields to Hensel semi-global fields without going through patching for vector spaces.

Before proceeding, we first introduce \emph{intermediate patches} to bridge the gap between patches and Hensel patches.
Let~$A$ be a finite type~$T$-subalgebra of~$\widehat{T}$ and let~$\mathcal{X}_A \coloneqq \mathcal{X}\times_T A$. 
The function field of~$\mathcal{X}_A$ is~$F_A \coloneqq \Frac(F \otimes_T A)$. Then~$X$ embeds into the closed fibre of~$\mathcal{X}_A$. Indeed, there is a morphism~$\mathcal{X}_{\widehat{T}} \to \mathcal{X}$ with a section above the closed fibre~$X$. By composing this section with the morphism~$\mathcal{X}_{\widehat{T}} \to \mathcal{X}_A$ induced by~$A \hookrightarrow \widehat{T}$, we obtain a morphism that maps~$X$ into the closed fibre of~$\mathcal{X}_A$. 
Let~$P$ be a closed point on~$X$ and~$U$ a connected open subset of~$X$ such that~$P$ is contained in~$\bar{U}$. Let~$R_{A, U} = \{f \in F_A\mid f \textnormal{ is regular on } U\}$ and let~$I$ be the ideal defining~$U$. Let~$R_{A, P}^h \coloneqq \mathcal{O}_{\mathcal{X}_A, P}^h$ and let~$R_{A, U}^h$ be obtained by taking the henselization of~$(R_{A,U}, I)$. Let~$F_{A, P}^h, F_{A, U}^h$ be their fraction fields, respectively. Let~$\wp$ be a height-one prime ideal of~$R_{A, P}^h$ that contains the uniformizer~$t$. Let~$R_{A, \wp}^h$ be henselization of the localization of~$R_P^h$ at~$\wp$, and let~$F_{A, \wp}^h$ denote its fraction field. 
Note that when~$A = T$,~$R_{A, P}^h = R_P^h$,~$R_{A, U}^h = R_U^h$, and~$R_{A, \wp}^h = R_{\wp}^h$. Moreover, when~$A = \widehat{T}$, if we complete~$R_{\widehat{T}, P}^h$,~$R_{\widehat{T}, U}^h$,~$R_{\widehat{T}, \wp}^h$ at the ideals that define~$P$,~$U$, and~$\wp$ respectively, we can recover the patches~$\widehat{R}_P$,~$\widehat{R}_U$ and the branch~$\widehat{R}_{\wp}$.

\begin{remark}\label{normal_patch}
    \begin{enumerate}
        \item Note that the closed fibre of~$\mathcal{X}_A$ is in general not~$X$. For example, take~$A = T[y]$, where~$y\in \widehat{T}$ is transcendental over~$T$ and then the closed fibre of~$\mathcal{X}_A$ is the~$y$-line over~$X$.
        \item The rings~$R_{\widehat{T},P}$ and~$R_{\widehat{T},U}$ are normal. Indeed, since~$\mathcal{X}\times \widehat{T}$ is normal, by definition, local rings at a closed point~$R_{\widehat{T},P}$ is normal; moreover, the maximal ideals of~$R_{\widehat{T},U}$ corresponds to closed points of~$U$, so local rings of~$R_{\widehat{T}, U}$ at those closed points are normal, whence~$R_{\widehat{T},U}$ is normal.
    \end{enumerate}
\end{remark}

\subsection{Patching for torsors under finite constant group}
Let~$(E, E_1, E_2, E_0)$ be a quadruple of fields and let~$G$ be a linear algebraic group over~$E$. A patching problem for~$G$-torsors asks when the map
\begin{equation*}
    \beta \colon G\tors(E) \to G\tors(E_1)\times_{G\tors(E_0)} G\tors(E_2)
\end{equation*} given by base change is an equivalence of categories. Here,
$G\tors(E)$ denotes the category of~$G$-torsors over~$E$. \cite[Theorem 2.3]{HHK15} shows that any patching problem for torsors has a solution over~$(\widehat{F}, F_P, F_U, F_{\wp})$. We will extend this result to Hensel semi-global fields and Hensel patches in the case of finite and constant group~$G$.

\begin{theorem}\label{patching_torsors} 
 The functor~$\beta$ defined above is an equivalence of categories over the quadruple of fields~$(F, F_P^h, F_U^h, F_{\wp}^h)$ when~$G$ is a finite constant group over~$F$.
\end{theorem}

\noindent Before proving \Cref{patching_torsors}, we first prove a compatibility result.

\begin{proposition}\label{compatibility_torsor} Let notation be as in \Cref{subsec:patching}.
     Let~$G$ be a linear algebraic group over a Hensel semi-global field~$F$. For any~$\xi \in \mathcal{P} \cup \mathcal{U}$, the natural map
     \begin{equation*}
         i \colon H^1(F_{\widehat{T},\xi}^h, G) \to H^1(F_{\xi}, G)
     \end{equation*} is injective. 
\end{proposition}

\begin{proof}
    We first show that~$i$ has trivial kernel. 
    Note that~$H^1(F_{\widehat{T}, \xi}^h, G)$ classifies isomorphism classes of~$G$-torsors over~$F_{\widehat{T}, \xi}^h$, and the natural map~$i$ is given by base change. 
    Suppose the~$G$-torsor~$Z$ becomes trivial after base changing to~$F_{\xi}$, so~$Z(F_{\xi}) \neq \emptyset$. Suppose~$Z$ is defined by a system of polynomials~$f_1, \cdots, f_n \in F_{\widehat{T},\xi}^h[x_1, \cdots, x_m]$. 
    By assumption there exists~$(a_1/b_1, \cdots, a_m/b_m) \in F_{\xi}^m$ such that~$f_s(a_1/b_1, \cdots, a_m/b_m) = 0$ for all~$s$, where~$a_j, b_j \in \widehat{R}_{\xi}$. 
    We can clear the denominators in~$f_s$ by multiplying by~$b_1^{e_1} \cdots b_m^{e_m}$ for some powers~$e_j$ so that~$g_s = y_1^{e_1} \cdots y_m^{e_m}f_s \in R_{\widehat{T},\xi}^h[y_1, \cdots, y_m, x_1, \cdots, x_m]$ has a solution~$(b_1, \cdots, b_m, a_1, \cdots, a_m) \in \widehat{R}_{\xi}^{2m}$. 
    Since~$R_{\widehat{T},\xi}$ is normal (see part 2 of \Cref{normal_patch}), by \Cref{theorem:artin_henselian_pair}, there exists~$(b_1', \cdots, b_m', a_1', \cdots, a_m') \in (R_{\widehat{T}, \xi}^h)^{2m}$ that is a solution to the system of polynomial equations~$g_1 = 0, \dots, g_n = 0$. 
    Thus,~$(a_1'/b_1', \cdots, a_m'/b_m')$ is a solution to~$f_s = 0$ for all~$s$. Hence,~$Z(F_{\widehat{T},\xi}^h) \neq \emptyset$, so $Z$ is a trivial $G$-torsor over $F_{\widehat{T},\xi}^h$. 

    Using a twist argument, we conclude that $i$ is indeed injective: suppose~$\alpha, \beta \in H^1(F_{\widehat{T},\xi}^h, G)$ such that~$i(\alpha) = i(\beta) \in H^1(F_{\xi}, G)$. Let~$\tau \in Z^1(F_{\widehat{T},\xi}^h, G)$ be a 1-cocycle that represents~$\alpha$ and let~$G^{\tau}$ be the twist of~$G$ by~$\tau$. By \cite[Proposition I.5.3.35]{serre97}, there is a bijection between~$H^1(F_{\widehat{T},\xi}^h, G)$ and~$H^1(F_{\widehat{T},\xi}^h, G^{\tau})$, where~$\tau \in H^1(F_{\widehat{T},\xi}^h, G)$ corresponds the neutral element of~$H^1(F_{\widehat{T},\xi}^h, G^{\tau})$. Then it suffices to show that the map $i' \colon H^1(F_{\widehat{T},\xi}^h, G^{\tau}) \to H^1(F_{\xi},G^{\tau})$ has trivial kernel. This was done in the previous paragraph.
\end{proof}

\begin{remark} Note that the above proposition remains true if we replace~$F_{\widehat{T},\xi}^h$ by~$F_{\xi}^h$, since the completion of~$R_{T,\xi}$ is also~$\widehat{R}_{\xi}$.
\end{remark}

\begin{proposition}\label{prop:morphism}
    Let~$E$ be a field and let~$L$ be a field extension of~$E$. Let~$G$ be a linear algebraic group over~$E$, and let~$X_1, X_2$ be~$G$-torsors over~$E$ that are isomorphic. Then the natural map~$i \colon \Mor_E(X_1, X_2) \to \Mor_{L}(X_{1, L}, X_{2,L})$ is injective. The map~$i$ is also surjective if~$G$ is finite and constant.
\end{proposition}

\begin{proof}
    The automorphism group (as $G$-torsors)~$\Uaut(X_i)$ is represented by a linear algebraic group~$H_i$ that is a twisted form of~$G$. 
    Therefore,~$\Aut_E(X_i)  \to \Aut_L(X_{i,L})$ is just~$H_i(E) \hookrightarrow H_i(L)$, which is injective. If~$G$ is finite and constant, then~$H_i(L) = H_i(E)$, so~$\Aut_E(X_i)  \to \Aut_L(X_{i,L})$ is also bijective.
    
    Let~$\phi_1, \phi_2 \in \Mor_E(X_1, X_2)$. Suppose~$i(\phi_1) = i(\phi_2)$. Then~$\phi_1 \circ \phi_2^{-1}, \id \in \Aut_E(X_2)$ have the same image under~$i$, so~$\phi_2 = \phi_1$. Therefore,~$i$ is injective. For the surjectivity, we  consider~$\Mor_L(X_{1,L}, X_{2,L})$ as a (left)~$H_1$-torsor. For any~$\phi \in \Mor_L(X_{1,L}, X_{2,L})$,  we may take~$\psi \in \Mor_E(X_1, X_2)$ and find~$h\in H_1$ such that~$i(\psi) = h.\phi$. Therefore,~$i(h^{-1}.\psi) = \phi$.
\end{proof}

\subsection{Proof of Theorem \ref{patching_torsors}}
Now we are ready to prove \Cref{patching_torsors}.
\begin{proof}[Proof of Theorem~\ref{patching_torsors}] Consider the patching problem for $G$-torsors $(X_P, X_U;\phi)$ over $(F, F_P^h, F_U^h, F_{\wp}^h)$.
    We show that there exists a~$G$-torsor~$X$ that induces~$X_P$ and~$X_U$, and the map~$\phi$.

    We define a functor 
    \begin{align*}
        \mathcal{F} \colon (T\subalg \widehat{T}) &\longrightarrow (\textnormal{Sets}),\\
        A &\longmapsto \mathcal{F}(A)
    \end{align*} where~$\mathcal{F}(A) = \{ G_{F_A} \textnormal{ -torsors } X_A \mid X_A \textnormal{ satisfies conditions (\ref{eq1}}), (\ref{eq2}), (\ref{eq3}) \textnormal{ defined below}\}$, and for any ring homomorphism 
   ~$f \colon A \to A'$,~$\mathcal{F}(f)(X_A) \coloneqq X_A \times_{F_A} F_{A'}$.

\begin{equation}\label{eq1}
    \exists \varphi_P \colon X_A \times_{F_A} F_{A, P}^h \overset{\cong}{\longrightarrow} X_P \times_{F_P^h} F_{A, P}^h,
\end{equation}
\begin{equation}\label{eq2}
    \exists \varphi_{U} \colon X_A \times_{F_A} F_{A, U}^h \overset{\cong}{\longrightarrow} X_U \times_{F_U^h} F_{A, U}^h,
\end{equation}
\begin{equation}\label{eq3}
    \begin{tikzcd}
        &X_A \times_{F_A} F_{A, \wp}^h \arrow{r}{\varphi_P \times_1} \arrow{d}{\cong} & X_P \times_{F_P^h} F_{A, P}^h \times_{F_{A, P}^h} F_{A, \wp}^h  \arrow{r}{=} \arrow[d, phantom, "\circlearrowright"]{=} &X_P \times_{F_P^h} F_{\wp}^h \times_{F_{\wp}^h} F_{A, \wp}^h \arrow{d}{\phi \times 1}\\
        &X_A \times_{F_A} F_{A, \wp}^h \arrow{r}{\varphi_U \times_1} &X_U \times_{F_U^h} F_{A, U}^h \times_{F_{A, U}^h} F_{A, \wp}^h \arrow{r}{=} &X_U \times_{F_U^h} F_{\wp}^h \times_{F_{\wp}^h} F_{A, \wp}^h
    \end{tikzcd} 
\end{equation}
\noindent One may check that~$\mathcal{F}$ defines a functor.

Now we check that~$\mathcal{F}$ is locally of finite presentation. 
It suffices to check that for any~$B = \varinjlim_i A_i$, where~$B$ is a finite type~$T$-subalgebra of~$\widehat{T}$,~$\mathcal{F}(B) = \varinjlim_i \mathcal{F}(A_i)$ by \Cref{artin_modified}. 
By \Cref{linearly_disjoint},~$F_{A_i}$ a field for each~$i$. It is easy to check that~$F_B = \varinjlim_i F_{A_i}$.
The map~$\phi_{i} \colon A_i \rightarrow B$ induces~$\mathcal{F}(\phi_{i}) \colon \mathcal{F}(A_i) \rightarrow \mathcal{F}(B)$. Let~$\Phi \coloneqq \varinjlim_i \mathcal{F}(\phi_i)$.
Take~$X_{B} \in \mathcal{F}(B)$. Since~$X_{B}$ is an affine variety,  it is defined by finitely many polynomials~$f_1, \dots, f_m \in F_B[x_1, \dots, x_n]$. Since~$F_B = \varinjlim_i F_{A_i}$, there exists~$i$ such that~$f_s \in F_{A_i}[x_1, \dots, x_n] \subseteq F_{A_i}[x_1, \dots, x_n]$ for all~$s$. Then the~$F_{A_i}$-variety~$\Spec(F_{A_i}[x_1, \dots, x_n]/(f_1, \dots, f_s))$ is isomorphic to~$X_B$ after base changing to~$F_B$. Therefore,~$\Phi$ is surjective. Let~$\psi_i \colon \mathcal{F}(A_i) \to \mathcal{F}(B)$. If~$\Phi(Y) = \Phi(Z)$, where~$Y, Z \in \varinjlim_i \mathcal{F}(A_i)$, then there exists~$i, j$ and such that~$Z_i \in \mathcal{F}(A_i), Y_j \in \mathcal{F}(A_j)$,~$\psi_i(Z_i) = Z, \psi_j(Y_j) = Y$ and~$Z_i \times_{F_{A_i}} F_B \cong Y_j \times_{F_{A_j}} F_B$. 
Therefore, there must exists an index~$k$ such that~$Z_i$ and~$Y_j$ are isomorphic after base changing to~$F_{A_k}$, whence~$Z$ and~$Y$ are isomorphic. 
Then~$\Phi$ is injective.

The isomorphism ~$\varphi_P$ is explicitly given by polynomials over~$F^h_{\widehat{T},P}$. We may clear the denominators and assume they are defined over~$R^h_{\widehat{T}, P}$, which is the direct limit of the coordinate rings of \'etale neighborhoods of~$P$ on~$\mathcal{X}_{\widehat{T}}$. Therefore, the coefficients must lie in the coordinate ring of some \'etale neighborhood of~$P$, which is a finite type~$\widehat{T}$-algebra. Then the coordinate ring must be defined by polynomials with coefficients in~$A_j$ for some~$j$, whence the isomorphism~$\varphi_P$ is defined by polynomials with coefficients in~$F_{A_j}$. Similarly, we can show that~$\varphi_U$ is given by polynomials with coefficients in~$F_{A_k}$ for some~$k$. Then (\ref{eq3}) is also automatically defined by algebraic relations. Hence, we have shown that~$\mathcal{F}$ meets the requirement of Proposition~\ref{artin_modified}.

To show that~$\mathcal{F}(\widehat{T}) \neq \emptyset$, we must show that the patching problem~$\mathcal{P}^h \coloneqq (X_{P,F_{\widehat{T},P}^h}, X_{U,F_{\widehat{T},U}^h}; \phi \times F_{\widehat{T},\wp}^h)$ of~$G$-torsors over~$(\widehat{F}, F_{\widehat{T}, P}^h, F_{\widehat{T}, U}^h, F_{\widehat{T}, \wp}^h)$ has a solution. 
Note that this patching problem induces a patching problem~$\mathcal{P}^{\wedge}\coloneqq (X_{P,F_P}, X_{U,F_U}; \phi \times F_{\wp})$ over~$(\widehat{F}, F_P, F_U, F_{\wp})$, which has a unique solution ~$X'$ by \cite[Theorem 2.3]{HHK15}. 

We now show that~$X'$ must be a solution to the patching problem~$\mathcal{P}^h$. Let $X_{\xi}'\coloneqq X'\times_{\widehat{F}} F_{\widehat{T}, \xi}^h$, where $\xi \in \{P, U\}$. On the level of objects, the~$G_{F_{\widehat{T},\xi}^h}$-torsors~$X_{\xi} \times_{F_{\xi}^h} F_{\widehat{T}, \xi}^h$ and~$X_{\xi}'$ both induce the~$G_{F_{\xi}}$-torsor~$X_{\xi,F_{\xi}}$. 
By \Cref{compatibility_torsor},~$X_{\xi} \times_{F_{\xi}^h} F_{\widehat{T}, {\xi}}^h \cong X_{\xi}'$ as~$G_{F_{\widehat{T}, {\xi}}^h}$-torsors.  
Now we check that~$X'$ is also compatible with the given morphism ~$\phi \times F_{\widehat{T},\wp}^h$. 
Note that~$X'$ is also a solution to the patching problem~$(\mathcal{P}^{h})'\coloneqq (X_P', X_U'; \textnormal{id})$. Therefore, by \cite[Theorem 2.3]{HHK15}, over~$(\widehat{F}, F_P, F_U, F_{\wp})$, the patching problems~$(\mathcal{P}^{\wedge})'\coloneqq (X_{P,F_P}', X_{U,F_U}';\textnormal{id})$ and~$\mathcal{P}^{\wedge}$ are isomorphic, so there exist isomorphisms~$\phi_{\xi} \colon X_{\xi,F_{\xi}}' \to X_{\xi,F_{\xi}}$, where $\xi \in \{P,U\}$ such that~$\phi \times F_{\wp} = \phi_P^{-1} \circ \phi_U$. Moreover, by \Cref{prop:morphism},~$\phi_{\xi}$ defines an isomorphism over~$F_{\widehat{T},\xi}^h$. 
Hence, the patching problems $\mathcal{P}^h$ and $(\mathcal{P}^h)'$ are isomorphic, so~$X'$ is also a solution to ~$\mathcal{P}^h$, whence~$\mathcal{F}(\widehat{T}) \neq 0$, so~$\mathcal{F}(T) \neq 0$ by \Cref{artin_modified}. Hence, there exists a~$G_F$-torsor~$X$ that induces~$X_P$,~$X_U$ and~$\phi$. 
Since by our construction,~$X\times_F \widehat{F}$ must also be a solution to~$\mathcal{P}^{\wedge}$, which has a unique solution~$X'$ as explained above, so $X \times_F \widehat{F} \cong X'$.

It remains to show that the functor $\beta$ is fully faithful. Let~$G$-torsors~$X$,~$X'$ be solutions to the patching problems~$\mathcal{P}$ and~$\mathcal{P}'$ over~$(F, F_P^h, F_U^h, F_{\wp}^h)$. Let~$\widehat{\mathcal{P}}$ and~$\widehat{\mathcal{P'}}$ denote the induced patching problems over~$(\widehat{F}, F_P, F_U, F_{\wp})$. By \cite[Theorem 2.3]{HHK15}, there natural map~$\Mor(X \times_F \widehat{F}, X' \times_F \widehat{F}) \to \Mor(\widehat{\mathcal{P}}, \widehat{\mathcal{P'}})$ is bijective. 
By \Cref{prop:morphism}, this also implies the natural map~$\Mor(X, X') \to \Mor(\mathcal{P}, \mathcal{P}')$ is bijective. 
\end{proof}

Hence, we obtain the following patching result as an immediate corollary.
\begin{corollary} \label{cor:patching}The following functor given by base change
    \begin{equation*}
    \beta \colon G\Galalg(F) \longrightarrow G\Galalg(F_P^h) \times_{G\Galalg(F_{\wp}^h)} G\Galalg(F_U^h)
\end{equation*} is an equivalent of categories. In particular, when~$G = S_n$, the assertion holds for the category of \'etale algebras of degree~$n$.
\end{corollary}
\begin{proof} Since the following composition of maps 
    \begin{equation*}
        \mathcal{F} \colon G\tors(F) \to H^1(F, G) \to G\Galalg(F)
    \end{equation*} is a natural bijection, the assertion follows from \Cref{patching_torsors}.
    When~$G = S_n$,~$H^1(F, S_n)$ classifies \'etale~$F$-algebras of degree~$n$, so the assertion follows.
\end{proof}

\begin{remark} 
    In the case of $S_n$-torsors, viewed as \'etale algebras, \Cref{cor:patching} can also be obtained using a geometric approach by considering the spectrum of an \'etale algebra over ~$F$ as the generic fibre of some branched cover of the given model $\mathcal{X}$. See the author's Ph.D. thesis (\cite[Section 3.4.2]{wang24}).  
\end{remark}

\subsection{Some preparatory results from patching}

In this section, we prove a few refined results about patching which will be needed in \Cref{sec:period-index}. 

Let $F$ be a Hensel semi-global field and let $\mathcal{X}$ be a model of $F$ with closed fibre $X$ as in the beginning of \Cref{subsec:patching}. Let $X^{\red}$ denote the reduced closed fibre of $\mathcal{X}$. 
Let~$\mathcal{P} \subseteq X^{\red}$ be the set of all singular points of ~$X^{\red}$. Let~$\mathcal{U}$ be the set of connected components of the complement~$X^{\red}\backslash \mathcal{P}$.

\begin{lemma}\label{etale_branch}
    Let~$P$ be a closed point of the closed fibre of~$\mathcal{X}$ and let~$\wp_i$ denote the branches of~$X$ at~$P$. For each~$i$, suppose that~$E_{\wp_{i}}^h$ is a degree~$d$ {\'e}tale~$F_{\wp_i}^h$-algebra. Then there exists an {\'e}tale~$F_P^h$-algebra~$E_P^h$ of degree~$d$ such that~$E_{\wp_i}^h \cong E_{P}^h \otimes_{F_P^h} F_{\wp_i}^h$ for all~$i$.
\end{lemma}

\begin{proof} 
This proof is a henselian analogue of the proof of \cite[Lemma 3.1]{HHK23}. 
    Let~$t_i$ denote the uniformizer of~$\wp_i.$ Since~$F_{\wp_i}^h$ is infinite, by \cite[Corollary 4.2(d)]{FR17}, any \'etale~$F_{\wp_i}^h$-algebra of finite rank~$E_{\wp_i}$ can be generated a primitive element. Let~$f_{\wp_i} \in F_{\wp_i}^h[x]$ be the monic polynomial of the primitive element. We may assume that~$f_{\wp_i} \in R_{\wp_i}^h[x]$ by multiplying the primitive element by a power of ~$t_i$. 
    Write~$f_{\wp_i} = f_{1i} \cdots f_{ni}$, where each~$f_{ji}$ is irreducible over ~$R_{\wp_i}^h$. 
    By applying \cite[Theorem 1]{Brink06} to each~$f_{ji}$, there exists a positive integer~$p_{ji}$ such that any monic polynomial~$g \in R_{\wp_i}^h[x]$ that satisfies~$g \equiv f_{ji} \mod t_i^{p_{ji}}$ defines the same algebraic extension of~$F_{\wp_i}^h$ as~$f_{ji}$. Take~$p_i$ to be the maximum of the exponents~$p_{ji}$ for~$ 1 \leq j \leq n$.
    By Hensel's Lemma \cite[Theorem 8]{Brink06}, there is an integer~$m_i$ such that for any monic polynomial~$g_{\wp_i} \in R_{\wp_i}^h[x]$ that satisfies~$g_{\wp_i} \equiv f_{\wp_i} \mod t_i^{m_i}$, any irreducible factor ~$g_{ji}$ of~$g_{\wp_i}$ satisfy~$g_{ji} \equiv f_{ji} \mod t_i^{p_i}$ for all~$1\leq j \leq n$.

    By the Weak Approximation Theorem (\cite[Theorem VI.2.1]{bourbaki}),~$F_P^h$ is dense in~$\prod F_{\wp_i}^h$. Therefore, there exists~$f \in F_P^h[x]$ such that~$f \equiv f_{\wp_i} \mod t_i^{m_i}.$ Moreover, by our choice of~$m_i$, the irreducible factors~$h_{ji}$ of~$f$ over~$R_{\wp_i}^h$ are in bijection with the irreducible factors~$f_{ji}$ of~$f_{\wp_i}$, and~$h_{ji} \equiv f_{ji} \mod t^{p_i}$. Furthermore, by the argument in the previous parapgrah, our choice of~$p_i$ shows that each~$h_{ji}$ defines the same algebraic extension of~$F_{\wp_i}^h$ as~$f_{ji}$, whence~$f$ and~$f_{\wp_i}$ define the same {\'e}tale algebras over~$F_{\wp_i}^h$. Hence, the {\'e}tale~$F_P^h$-algebra~$E_P^h$ defined by~$f$ induces~$E_{\wp_i}$ by base change.
\end{proof}

\begin{lemma}\label{etale_big_patch}
    Suppose that we are given a degree~$d$ {\'e}tale~$F_U^h$-algebra~$E_U^h$ for all~$U \in \mathcal{U}.$ Then there exists an {\'e}tale~$F$-algebra~$E^h$ of degree~$d$ such that~$E^h \otimes_F F_U^h \cong E_U^h$.
\end{lemma}

\begin{proof}
    For each point~$P \in \mathcal{P}$, each branch~$\wp$ at~$P$ lies on the closure of a unique~$U \in \mathcal{U}$. By Lemma~\ref{etale_branch}, there exists an {\'e}tale~$F_P^h$-algebra~$E_P^h$ of degree~$d$ such that~$E_{P}^h \otimes_{F_P^h} F_{\wp}^h$ is isomorphic to the \'etale~$F_{\wp}$-algebra~$E_U^h \otimes_{F_U^h} F_{\wp}^h$. Fix this isomorphism~$\phi$. Then we obtain a patching problem for \'etale algebras~$(E_P^h, E_U^h;\phi)$. By \Cref{cor:patching}, this patching problem has a solution, i.e., an \'etale~$F$-algebra~$E^h$, that has the desired property of the statement.
\end{proof}

\begin{lemma}\label{etale_small_patch}
    Suppose that for all~$P \in \mathcal{P}$, we are given an {\'e}tale~$F_P^h$-algebra~$E_P^h$ of degree~$d$ that is coprime to char$(k)$. Assume that the integral closure of~$R_P^h$ in~$E_P^h$ is unramified over~$R_P^h$. Then there exists an {\'e}tale~$F$-algebra~$E$ such that~$E \otimes_F F_P^h \cong E_P^h$ for all~$P \in \mathcal{P}$.
\end{lemma}

\begin{proof}
    The \'etale~$F_P$-algebra~$E_P \coloneqq E_P^h \otimes_{F_P^h} F_P$ is necessarily of degree~$d$. The integral closure of~$\widehat{R}_P$ in~$E_P$ is unramified as the Jacobian matrix is of maximal rank. 
    It was shown in \cite[Lemma 3.3]{HHK23} that under the same assumptions, there exists an \'etale~$\widehat{F}$-algebra~$\widehat{E}$ such that~$\widehat{E}\otimes_{\widehat{F}} F_P \cong E_P$ for all $P\in \mathcal{P}$. By an argument similar the one in the proof of \Cref{patching_torsors}, the functor
    \begin{equation*}
        \mathcal{F}(A) = \{\text{\'etale~$F_A$-algebra~$L_A$}\mid L_A \otimes_{F_A} F_{A, P}^h \cong E_{P}^h \otimes_{F_{P}^h} F_{A, P}^h \textnormal{ for all $P\in \mathcal{P}$}\}
    \end{equation*} meets the requirement of Proposition~\ref{artin_modified}.  Since both \'etale~$F_{\widehat{T}, P}^h$-algebras~$\widehat{E} \otimes_{\widehat{F}} F_{\widehat{T},P}^h$ and~$E_P^h \otimes_{F_{P}^h} F_{\widehat{T},P}^h$ are isomorphic after base changing to~$F_P$, they must already be isomorphic over~$F_{\widehat{T},P}^h$ by \Cref{compatibility_torsor}. Therefore,~$\widehat{E} \in \mathcal{F}(\widehat{T})$. By Proposition~\ref{artin_modified}, there exists an \'etale ~$F$-algebra~$E$ as required in the theorem.
\end{proof}

\section{A local-global principle with respect to patches}\label{sec_lgp}

In this section we study local-global principles for principal homogeneous spaces and higher degree Galois cohomology groups over Hensel semi-global fields, by comparing to the known results for semi-global fields.

\subsection{A comparison between two cohomology sets}
Let~$T$ be an excellent henselian discrete valuation ring and let~$\widehat{T}$ be its completion. Let~$K, \widehat{K}$ denote the fraction fields of~$T, \widehat{T}$, respectively. Let~$F$ be a finitely generated field extension of~$K$ for which~$K$ is algebraically closed in~$F$. Under this assumption,~$F \otimes_T \widehat{T}$ is an integral domain and therefore~$\widehat{F} \coloneqq \Frac(F \otimes_T \widehat{T})$ is a field (see Proposition~\ref{linearly_disjoint}). As an example, one may take~$F$ to be a Hensel semi-global field, and~$\widehat{F}$ to be the corresponding semi-global field.

\begin{proposition}\label{hensel_inj}
    Let notation be as above. Let~$G$ be a linear algebraic group over~$F$. The natural map
    \begin{equation*}
        i \colon H^n(F, G) \longrightarrow H^n(\widehat{F}, G)
    \end{equation*} has trivial kernel for~$n \geq 1$ if~$G$ is abelian and~$n = 1$ if~$G$ is nonabelian.
\end{proposition}

\begin{proof}
    \textit{Case of~$n = 1$}: Note that~$H^1(F, G)$ classifies isomorphism classes of~$G$-torsors. Let~$Z$ represent a class~$\alpha \in H^1(F, G)$, for which~$i(\alpha)$ is the zero class in~$H^1(\widehat{F}, G)$. Therefore,~$Z_{\widehat{F}}$ is a trivial~$G_{\widehat{F}}$-torsor. By \cite[Proposition 3.19]{HHKCPS20},~$Z(F) \neq \emptyset$, so~$Z$ represents the trivial class in~$H^1(F, G)$.

    \textit{Case of~$n > 1$}: Let~$\alpha \in  H^n(F, G)$ and~$\widehat{\alpha} \coloneqq i(\alpha) = H^n(\widehat{F}, G).$ Suppose that~$\widehat{\alpha}$ is cohomologous to the zero class~$[0] \in H^n(\widehat{F}, G).$ We will show that~$\alpha \sim [0] \in H^n(F, G)$. Let~$d^{n-1} \colon C^{n-1}(\widehat{F}, G) \rightarrow C^{n}(\widehat{F}, G)$ denote the coboundary map. By assumption,~$\widehat{\alpha}$ is a cocyle in~$Z^n(\widehat{F}, G)$ that is the coboundary of a cochain in~$C^{n-1}(\widehat{F}, G)$, so there exists~$\widehat{\beta} \in C^{n-1}(\widehat{F}, G)$ such that~$d^{n-1}(\widehat{\beta}) = \widehat{\alpha}.$

    Let~$\widehat{T} = \varinjlim_i A_i$, where each~$A_i$ is a~$T$-subalgebra of~$\widehat{T}$. Then~$F_{A_i} \coloneqq \Frac(F \otimes_T A_i)$ is a field by Proposition~\ref{linearly_disjoint}. Define the following functor:
    \begin{align*}
        \mathcal{F}_{\alpha} \colon (T\subalg \widehat{T}) & \longrightarrow (\textnormal{Sets}) \\
        A                                                      & \longmapsto \mathcal{F}_{\alpha}(A),
    \end{align*} where~$$\mathcal{F}_{\alpha}(A) = \{\beta \in C^{n-1}(F_A, G)\mid d^{n-1}(\beta) = \alpha_{F_A} \in C^{n}(F_A, G)\}.$$ 

    \noindent Now we show that~$\mathcal{F}_{\alpha}$ satisfies the requirement of Proposition~\ref{artin_modified}, i.e., we show that~$\mathcal{F}_{\alpha}(\widehat{T}) = \varinjlim_i(\mathcal{F}_{\alpha}(A_i))$. First, the map~$A_i \to \widehat{T}$ automatically induces a map~$\mathcal{F}_{\alpha}(A_i) \to \mathcal{F}_{\alpha}(\widehat{T})$. Conversely, to show that~$\mathcal{F}(\widehat{T}) \subseteq \varinjlim_i(\mathcal{F}_{\alpha}(A_i))$, it suffices to show that there exists~$i$ such that~$\alpha_{F_{A_i}} \sim 0 \in H^i(F_{A_i}, G)$. Since~$\widehat{\alpha}$ is trivial, there exists a a finite Galois extension~$L/\widehat{F}$ with Galois group~$\Gamma$ such that~$\widehat{\alpha}$ is the coboundary in a cochain~$C^n(\Gamma, G(L))$. Note that~$G$ is of finite type over~$F$ and therefore defined by finitely many polynomials over~$F$. Moreover,~$L/\widehat{F}$ is defined by a polynomial~$f$ over~$\widehat{F}$ and we may assume the coefficients of~$f$ lie in~$F \otimes_T \widehat{T}$ by clearing the denominators. Therefore,~$\widehat{\alpha}$ is defined over the field extension of~$F$ obtained by adjoining these coefficients, and hence there exists some $i$ such that ~$\alpha_{F_{A_i}}$ vanishes. 
    By Proposition~\ref{artin_modified}, ~$\mathcal{F}_{\alpha}(T)$ is nonempty since $\widehat{\alpha} \in \mathcal{F}_{\alpha}(\widehat{T})$ by construction. Therefore, there exists~$\beta \in C^{n-1}(F, G)$,~$d^{n-1}(\beta) = \alpha.$ Hence,~$\alpha$ is cohomologous to the zero class~$[0]$ in~$H^n(F, G)$, whence~$\ker(i) = 0$.
\end{proof}

\subsection{A local-global principle with respect to patches} Let notation be as in Section~\ref{subsec:patching}. To study the local-global principle over~$F$, we compare it to that over~$\widehat{F}$.

\begin{theorem}\label{lgp_henselian}
    Let notation be as in the beginning of this section. The natural map
    \begin{equation*}
        H^n(F, G) \longrightarrow \prod_{P \in \mathcal{P}} H^n(F_P^h, G) \times \prod_{U \in \mathcal{U}} H^n(F_U^h, G)
    \end{equation*} has trivial kernel when
    \begin{equation*}
        H^n(\widehat{F}, G) \longrightarrow \prod_{P \in \mathcal{P}} H^n(F_P, G) \times \prod_{U \in \mathcal{U}} H^n(F_U, G)
    \end{equation*} has trivial kernel. Here~$n = 1$ if~$G$ is non-abelian and~$n \geq 1$ when~$G$ is abelian. In particular, 
    \begin{equation*}
        H^n(F, \mu_m^{\otimes r}) \longrightarrow \prod_{P \in \mathcal{P}} H^n(F_P^h, \mu_{m}^{\otimes r}) \times \prod_{U \in \mathcal{U}} H^n(F_U^h, \mu_{m}^{\otimes r})
    \end{equation*} has trivial kernel for~$n > 1$ when either of the following holds:
    \begin{enumerate}
        \item~$r = n - 1$ and~$m$ is not divisible by the characteristic of the residue field~$k$.
        \item~$m$ is not divisible by the characteristic of the residue field~$k$ and~$[F(\mu_m): F]$ is prime to~$m$; this is satisfied when~$m$ is prime or~$F$ contains a primitive~$m$-th root of unity.
    \end{enumerate}
\end{theorem}
\begin{proof} By functoriality of Galois cohomology, we obtain a commutative diagram
\begin{center}
    \begin{tikzcd}
        &H^n(F, G) \arrow{r}{} \arrow{d}{}&\prod_{P \in \mathcal{P}} H^n(F_P^h, G) \times \prod_{U \in \mathcal{U}} H^n(F_U^h, G)\arrow{d}{}\\
        &H^n(\widehat{F}, G) \arrow{r}{} &\prod_{P \in \mathcal{P}} H^n(F_P, G) \times \prod_{U \in \mathcal{U}} H^n(F_U, G)
    \end{tikzcd}
\end{center} for~$n = 0, 1$ if~$G$ is nonabelian and any~$n \geq 0$ if~$G$ is abelian. 
    The first claim follows from diagram chasing and Proposition~\ref{hensel_inj}. The second claim holds by \cite[Theorem 3.1.5, Corollary 3.1.6]{HHK14}.
\end{proof}

\noindent Note that we do not know what to expect if the local-global principle over~$\widehat{F}$ does not hold.

\section{The period-index problem for higher degree Galois cohomology groups}\label{sec:period-index}

As an application of patching of \'etale algebras and the local-global principle, we prove a generalized period-index bound for higher degree cohomology groups~$\cohom{F}{m}$. Such a bound was considered in \cite{HHK23} when~$F$ is a semi-global field and~$m$ is a prime number that is not the characteristic of~$F$. In this section, we generalize their results to Hensel semi-global fields. We will also consider an analogous bound when~$m$ is not prime in Section~\ref{sec_non_prime}.

\subsection{Splitting dimensions and uniform bounds for cohomology classes} Throughout this section, let~$F$ be a field and let~$m$ a positive integer that is coprime to the characteristic of~$F$. The following definitions from Section 2 of \cite{HHK23} also extend here.

\begin{definition}
  Fix a degree~$i$. Let~$\alpha \in \cohom{F}{m}$. A field extension~$L/F$ is a \emph{splitting field} for~$\alpha$ if the natural map~$\cohom{F}{m} \rightarrow \cohom{L}{m}$ maps~$\alpha$ to the trivial class. Similarly, let~$B$ be a collection of classes in~$\cohom{F}{m}$. A field extension~$L/F$ is a \emph{splitting field} for~$B$ if~$L$ splits all the elements in~$B$.
\end{definition}

\begin{definition}
    The \emph{index of a class}~$\alpha \in \cohom{F}{m}$, denoted by~$\operatorname{ind}(\alpha)$, is the greatest common divisor of~$[L: F]$, where~$L$ ranges over all splitting fields of~$\alpha$ of finite degrees. Similarly, the \emph{index of a subset}~$B \subseteq \cohom{F}{m}$, denoted by~$\operatorname{ind}(B)$, is the greatest common divisor of~$[L: F]$, where~$L$ ranges over all splitting fields of~$B$ of finite degrees.
\end{definition}

By \cite[Remark 2.2 and Lemma 2.3]{HHK23}, when~$m = l$ is a prime number, for any non-trivial~$\alpha \in \coho{F}$,~$\operatorname{ind}(\alpha)$ must be a power of~$l$. More generally, for~$B \subseteq \coho{F}$,~$\operatorname{ind}(B)$ must also be a power of~$l$. Therefore, we may also define the following notions.

\begin{definition}
    Fix a prime~$l$ and a degree~$i$. We define \emph{i-splitting dimension at l of F} as the minimal exponent~$n$ of~$l$, denoted by~$\operatorname{sd}_l^i(F)$, so that for all~$\alpha \in \coho{F}$, the index of~$\alpha$ divides~$l^n$. 
\end{definition}

\begin{definition}
    The \emph{stable i-splitting dimension at l of F}, denoted by~$\operatorname{ssd}_l^i(F),$ is the minimal exponent~$n$ so that~$\operatorname{sd}_l^i(L) \leq n$ for all finite field extensions~$L/F$. 
\end{definition}

\begin{definition}
    The \emph{generalized stable i-splitting dimension at l of F}, denoted by~$\operatorname{gssd}_l^i(F)$ is the minimal exponent~$n$ so that the index of~$B$ divides~$l^n$ for all finite field extensions~$L/F,$ and for all finite subsets~$B \subseteq \coho{F}$.
\end{definition}

The last two stronger forms of splitting dimensions provide measurements for stability of splitting dimensions under finite field extensions.

\subsection{Splitting unramified cohomology classes}

Let~$E$ be a field and let~$v$ be a discrete valuation on~$E$. Let~$\kappa(v)$ denote the residue field. For any~$l$ that does not equal to the characteristic of~$\kappa(v)$, there is a specialization map~$r_v \colon \coho{E} \rightarrow H^{i-1}(\kappa(v), \mu_l^{\otimes i-2})$ (see \cite[Section 7.9]{GMS03}). If~$\mathscr{X}$ is a regular integral scheme with function field~$E$ and~$\mathscr{X}^{(1)}$ is the set of codimension one points of~$\mathscr{X}$, then every~$x \in \mathscr{X}^{(1)}$ defines a discrete valuation~$v_x$ on~$E$ . A class~$\alpha \in \coho{E}$ is said to be \emph{unramified at}~$x$ if~$r_{v_x}(\alpha) = 0$. Moreover,~$\alpha$ is \emph{unramified on}~$\mathscr{X}$ if it is unramified at all~$x \in \mathscr{X}^{(1)}$. Let~$H^i_{\textnormal{nr}}(E, \mu_l^{\otimes i - 1})^{\mathscr{X}}$ denote the subgroup of unramified classes of~$\coho{E}$.

\begin{lemma}\label{R_eta_h}
    Let notation be as in Section~\ref{subsec:patching}. Let~$U \in \mathcal{U}$ and let~$\eta$ be the generic point of the irreducible component~$X_0$ of~$X$ that contains~$U$. Let~$\tilde{R}_{\eta}^h \coloneqq \varinjlim_{V \subseteq U} R_V^h$, where~$V$ runs over affine open subsets of~$U$. Let~$\tilde{F}_{\eta}^h$ be the fraction field of~$\tilde{R}^h_{\eta}$. (Note that~$\tilde{R}_{\eta}^h$ and~$\tilde{F}_{\eta}^h$ are the Hensel analogues to~$R_{\eta}^h$ and~$F^{h}_{\eta}$ defined in \cite[Section 3.2]{HHK14}.) Then~$\tilde{R}_{\eta}^h$ is a discrete valuation ring with residue field~$k(U)$ and~$\tilde{F}_{\eta}^h = \varinjlim_{V\subseteq U} F_V^h$.
\end{lemma}

Note that here~$k(U)$ denotes the function field of~$U$. It is also the residue field of~$\Spec(R_U^h)$ at the generic point~$\eta$ of~$U$. 

\begin{proof}[Proof of Lemma~\ref{R_eta_h}] We will follow the proof of \cite[Lemma 3.2.1]{HHK14}.
    For any~$V$,~$F_V^h \subseteq \tilde{F}^h_{\eta}$ and any element~$ a/b \in \tilde{F}_{\eta}^h$, we must have~$a, b \in R_V^h$ for some common~$V$. Therefore,~$\tilde{F}_{\eta}^h = \varinjlim_{V\subseteq U} F_V^h$. 
    
    We may view~$\eta$ as a prime ideal of~$R_V^h$. Then each field~$F_{V}^h$ is a discretely valued field with respect to the~$\eta$-adic valuation, and the~$\eta$-adic valuations are compatible among all the fields~$F_V^h$ for all~$V \subseteq U$, whence~$\tilde{F}_{\eta}^h$ is an~$\eta$-adic discretely valuated field. We then have to show that the valuation ring of~$\tilde{F}_{\eta}^h$ is exactly~$\tilde{R}_{\eta}^h$. 
    
    First, we note the fact that the~$t$-adic and the~$\eta$-adic valuations are equivalent since~$\eta$ is the radical of~$(t)$. It is easy to see that~$\tilde{R}_{\eta}^h$ is contained in the valuation ring of~$\tilde{F}_{\eta}^h$. 
    For the other direction of containment, let~$y \in \tilde{F}_{\eta}^h$ be an element that lies in the valuation ring with respect to~$\eta$, i.e.,~$v_{\eta}(y) \geq 0$. We write~$y = a/b$, where~$a, b \in R_V^h$ for some~$V \subseteq U$. Note that~$R_V^h$ is a Krull domain, since~$\widehat{R}_V$, the completion of~$R_V^h$ with respect to the ideal that defines~$V$ in~$\Spec(R_V)$, is a Krull domain (see \cite[Chapter 7 Proposition 16]{bourbaki}). Therefore, we may uniquely write ~$b = u\cdot\mathfrak{p}_1^{\alpha_1}\cdots \mathfrak{p}_k^{\alpha_k}$, where~$u$ is a unit in~$R_V^h$ and~$\alpha_i = v_{\mathfrak{p}_i}(b)$. Then the divisor associated to~$b$,~$\operatorname{div}(b)$, can be uniquely written as a linear combination prime divisors~$\alpha_1 \mathfrak{p}_1 + \dots + \alpha_n\mathfrak{p}_n$. Moreover, the loci of these prime divisors excluding the irreducible closed fibre~$V$ of~$\operatorname{Spec}(R_V^h)$ itself intersect~$V$ at finitely many points. By deleting these points from~$V$, we can assume that~$b$ is invertible in~$R_V^h[t^{-1}]$. Moreover, since~$v_{\eta}(a/b) \geq 0$,~$a/b$ does not have poles on ~$\operatorname{Spec}(R_V^h)$ and therefore lies in~$R_V^h \subseteq \tilde{R}_{\eta}^h$. Then we conclude that~$\tilde{R}_{\eta}^h$ is a valuation ring of~$\tilde{F}_{\eta}^h$. 

    Since the~$\eta$-adic valuations on~$R_V^h$ for all~$V \subseteq U$ are compatible and also induce the~$\eta$-adic valuation on~$\tilde{R}_{\eta}^h$, the maximal ideal~$\eta \tilde{R}^h_{\eta}$ of~$\tilde{R}^h_{\eta}$ is just~$\varinjlim_{V \subseteq U}\eta R_V^h$. Since~$R_V^h/\eta R_V^h = k(V) = k(U)$ for all~$V \subseteq U$, we must have~$\tilde{R}^h_{\eta}/\eta\tilde{R}^h_{\eta} \cong k(U)$.
\end{proof}

\begin{lemma}\label{delete}
    Let notation be as in Section~\ref{subsec:patching}. Let~$U \in \mathcal{U}$. Suppose that~$\alpha \in \coho{F_U^h}$ is unramified on~$\operatorname{Spec}(R_U^h)$. There exists a nonempty subset~$U'$ of~$U$ such that~$\alpha_{F_{U'}^h}$ is in the image of~$\coho{R_{U'}^h} \rightarrow \coho{F_{U'}^h}$.
\end{lemma}

\begin{proof}
    Let~$\tilde{R}_{\eta}^h \coloneqq \varinjlim_{V\subseteq U} R^h_V$ and~$\tilde{F}_{\eta}^h$ the fraction field of~$\tilde{R}^h_{\eta}$. By Lemma~\ref{R_eta_h},~$\tilde{F}_{\eta}^h = \varinjlim_{V\subseteq U} F_V^h,$ and~$\tilde{R}_{\eta}^h$ is a discrete valuation ring with residue field~$k(U)$. Since~$\alpha$ is unramified, so is its image~$\alpha_{\tilde{F}_{\eta}^h}$ by functoriality and the fact that ~$\tilde{F}_{\eta}^h~$ and~$F_U^h$ have the same residue field~$k(U)$.  Therefore, there exists~$\tilde{\alpha} \in \coho{\tilde{R}_{\eta}^h}$ such that~$\tilde{\alpha}$ is mapped to~$\alpha_{\tilde{F}^h_{\eta}}$ by \cite[Section 3.3]{CT95}. Moreover, by \cite[Theorem 09YQ]{stacks-project}, 
    \begin{align*}
        \coho{\tilde{R}^h_{\eta}} &= \varinjlim_{V \subseteq U} \coho{R_V^h}, \\ \coho{\tilde{F}^h_{\eta}} &= \varinjlim_{V \subseteq U} \coho{F_V^h}.
    \end{align*} Therefore, there exist~$V \subseteq U$ and~$\tilde{\alpha}' \in \coho{R_V^h}$ such that~$\tilde{\alpha}'$ is mapped to~$\tilde{\alpha}$. Both the triangle and the square in the following diagram are commutative by functoriality.
    \begin{center}
        \begin{tikzcd}
            &\coho{F_U^h} \arrow{r}{} \arrow{dr}{}&\coho{F_V^h} \arrow{d}{} &\coho{R_V^h}\arrow{l}{}\arrow{d}{}\\
            & &\coho{\tilde{F}^h_{\eta}} & \coho{\tilde{R}^h_{\eta}}  \arrow{l}{}
        \end{tikzcd}
    \end{center} Therefore,~$\alpha_{F_{V}^h}, \tilde{\alpha}_{{F_V^h}}'$ have the same image in~$\coho{\tilde{F}^h_{\eta}}$. By Lemma~\ref{R_eta_h} again, $\tilde{F}^h_{\eta} = \varinjlim_{W \subseteq V} F_W^h$,
    so we may find~$U'\subseteq V$ such that the image of~$\alpha_{F_{V}^h}, \tilde{\alpha}_{{F_V^h}}'$ under~$\coho{F_V^h} \to \coho{F_{U'}^h}$ coincide. Therefore,~$U'$ is as desired.
\end{proof}

\begin{proposition}\label{split_unramified}
    Let~$K$ be an excellent henselian discretely valued field with residue field~$k$. Let~$F$ be a one-dimensional function field over~$K$. Let~$l$ be a prime that does not divide the characteristic of~$k$.
    \begin{enumerate}[label=(\alph*)]
        \item Suppose that~$\alpha \in \coho{F}$ is unramified on a regular model~$\mathcal{X}$ of~$F$ and~$i \geq 1$. Then 
        \begin{equation*}
            \operatorname{ind}(\alpha)\mid l^{\ssd_l^i(k) + \ssd_l^i(k(x))}.
        \end{equation*}
        \item Suppose that~$B \subseteq \coho{F}$ consists of finitely many classes that are unramified on a regular model~$\mathcal{X}$ of~$F$ and~$i \geq 2$. Then 
        \begin{equation*}
            \operatorname{ind}(B)\mid l^{\gssd_l^i(k) + \gssd_l^i(k(x))}.
        \end{equation*}
    \end{enumerate}
\end{proposition}

Note that as explained right after \cite[Theorem 2.9]{HHK23}, part (b) of the theorem is not interesting in the case~$i = 1$. First note that for any infinite field $F$, ~$H^1(E, \Z/l\Z) = E^{\times}/(E^{\times})^l$ is infinite for any finite extension~$E/F$. Since a non-trivial~$\Z/l\Z$-torsor over~$E$ corresponds to a field extension over~$E$ that splits only over itself,~$\gssd_l^1(F)$ must be infinity. Moreover, for the same reason,~$\ind(\alpha) = l$ for any non-trivial class ~$\alpha \in H^1(E, \Z/l\Z)$, and therefore~$\ssd_l^1(F) = 1$.

\begin{proof} 
    If~$k$ is finite, by Proposition~\ref{hensel_inj} and \cite[Remark 4.3]{HHK23},~$\coho{F}$ vanishes when~$i > 1$, so we only have to consider the case when~$k$ is infinite.

    We first prove part (b). The proof closely mirrors the proof of \cite[Proposition 4.2]{HHK23} with necessary adaptions to Hensel semi-global fields. Let~$B \subseteq \nrcoho{F}$ be a collection of classes that are unramified on~$\mathcal{X}$. By \cite[Lemma 2.3]{HHK23}, it suffices to find a finite field extension~$L/F$ such that all~$\alpha_j \in B$ split over~$L$, and~$v_l([L: F]) \leq \gssd_l^i(k) + \gssd_l^i(k(x))$. 
    Let~$\mathcal{P} \subseteq X^{\red}$ be the set of all singular points of the reduced closed fibre~$X^{\red}$ of~$\mathcal{X}$. Let~$\mathcal{U}$ be the set of connected components of the complement~$X^{\red}\backslash \mathcal{P}$.

    We fix a~$U \in \mathcal{U}$. By Lemma~\ref{delete}, we can delete finitely many points from~$U$ and add them to~$\mathcal{P}$ so that~$(\alpha_j)_{F_U^h}$ is the image of some~$\tilde{\alpha}_j \in \coho{R_U^h}$ under the natural map~$\coho{R_U^h} \to \coho{F_U^h}$. As explained in Section~\ref{subsec:patching},~$U = \operatorname{Spec}(R_U^h/I)$. By \cite[09ZI]{stacks-project},~$\coho{R_U^h} \cong \coho{U}$, and we can therefore assume~$\tilde{\alpha}_j \in \coho{U}$. Let~$\bar{\alpha}_j$ denote the image of~$\tilde{\alpha}_j$ under~$\coho{U} \to \coho{k(U)}$, where~$k(U)$ denotes the function field of~$U$. Note that $k(U)$ is a finite field extension of $k(x)$. Then there exists a finite field extension~$l_U/k(U)$ that splits~$\bar{\alpha}_j$ for all~$j$ and~$v_l([l_U: k(U)]) \leq \gssd_l^i(k(x))$. Take the separable closure~$k(U)^s$ of~$k(U)$ inside~$l_U$. 
    The degree of the purely inseparable field extension $l_U/k(U)^s$ is a power of the characteristic of~$k$, which is coprime to~$l$. Then~$k(U)^s$ must split~$\bar{\alpha}_j$ for all~$j$ by \cite[Remark 2.2]{HHK23}. Take~$V$ to be the normalization of~$U$ in~$k(U)^s$, so~$k(V) = k(U)^s$. Therefore,~$\tilde{\alpha}_j \mapsto 0$ for all~$j$ under the composition of the following maps
    \begin{equation*}
        \coho{R_U^h} \cong \coho{U} \longrightarrow \coho{k(U)} \longrightarrow \coho{k(V)},
    \end{equation*} and~$v_l([k(V): k(U)]) = v_l([l_U: k(U)]) \leq \gssd_l^i(k(x))$.

    Since~$\varinjlim_{U' \subseteq U} V \times_U U' = \varinjlim_{V' \subseteq V} V$, by \cite[Theorem 09YQ]{stacks-project},
    \begin{equation*}
        \coho{k(V)} = \coho{V} = \varinjlim_{V'\subseteq V}\coho{V'} = \varinjlim_{U'\subseteq U} \coho{V \times_U U'}.    
    \end{equation*} Therefore, there exists~$U' \subseteq U$ such that~$\tilde{\alpha}_j$ is mapped to zero in~$\coho{V \times_U U'}$. By construction,~$k(V)/k(U)$ is separable, so~$V \rightarrow U$ is generically {\'e}tale. Moreover, we may possibly delete a finite number of closed points from~$U'$ so that~$V \times_U U'$, the pullback of~$U'$ along~$V \rightarrow U$, is finite and {\'e}tale over~$U'$. Let~$J$ be the ideal of~$R_{U'}^h$ that defines~$U'$ in~$\operatorname{Spec}(R_{U'}^h)$. Then~$(R_{U'}^h, J)$ is a henselian pair. Therefore, by \cite[Lemma 09ZL]{stacks-project}, there exists a finite {\'e}tale cover~$\operatorname{Spec}(Z_{U'}) \to \operatorname{Spec}(R_{U'}^h)$ of the same degree as the closed fibre~$V \times_U U' \to U$. The following commutative diagram (note that the horizontal isomorphisms follow from \cite[Theorem 09ZI]{stacks-project}),
    \begin{center}
        \begin{tikzcd}
            &\coho{R_U^h} \arrow{r}{} & \coho{R_{U'}^h}\arrow{d}{} \arrow[r, "\cong"]&\coho{U'} \arrow{d}{}\\
            & & \coho{Z_{U'}} \arrow[r, "\cong"] &\coho{V \times_U U'}
        \end{tikzcd}
    \end{center} shows that the class~$\tilde{\alpha}_j$ is mapped to zero in~$\coho{Z_{U'}}$ for each~$j$. Since~$V$ is reduced and irreducible,~$Z_{U'}$ is an integral domain. By pulling back along the generic fibre,~$E_{U'}$, the fraction field of $Z_{U'}$, splits~$\alpha_j$ for all~$j$. By construction, the degree of the extension~$E_{U'}/F_{U'}^h$ equals the degree of the \'etale cover~$\operatorname{Spec}(Z_{U'}) \to \operatorname{Spec}(R_{U'}^h)$. Hence, by the discussion above,~$[E_{U'}: F_{U'}^h] = [k(V): k(U)]$, whence~$v_l([E_{U'}: F_{U'}^h]) = v_l([k(V): k(U)]) \leq \gssd_l^i(k(x))$.

     Note that we obtained each~$U'$ by deleting finitely many closed points from the~$U$ that we start with. Now we add these closed points to~$\mathcal{P}$ and replace~$\mathcal{U}$ with the set of connected components of the complements of~$\mathcal{P}$ in~$X^{\red}$. Let~$d_1$ be the least common multiple for the degrees~$[E_{U'}: F_{U'}^h]$ for all~$U' \in \mathcal{U'}$. By construction,~$v_l(d_1) \leq \gssd_l^i(k(x))$. Now we may take the direct sum of an appropriate number of copies of~$E_{U'}$ for each~$U'$ so that we obtain {\'e}tale~$F^h_{U'}$-algebras~$L_{U'}$ of a common degree~$d_1$. By Lemma~\ref{etale_big_patch}, there exists an {\'e}tale~$F$-algebra~$L_1$ of degree~$d_1$ such that~$L_1 \otimes_F F^h_{U'} \cong L_{U'}$ for all~$U'\in \mathcal{U}$. 
     
    Let~$P \in \mathcal{P}$. Since~$\alpha_j \in \nrcoho{F}$ is unramified, its image~$\alpha_{j,F_P^h} \in \coho{F_P^h}$ is also unramified (since the residue field~$\kappa(P)$ of~$F_P^h$ is an algebraic extension of~$k$). By \cite[Theorem 9]{Sakagaito20}, there exists~$\alpha_{j,R_P^h} \in \coho{R_P^h}$ that maps to~$\alpha_{j,F_P^h}$ under the natural map~$\coho{R_P^h} \to \coho{F_P^h}$. By \cite[Corollaire 5.5]{SGA73},~$\coho{\kappa(P)} \cong \coho{R_P^h}$. Let~$\alpha_{j,\kappa(P)}$ denote the image of~$\alpha_{j, R_P^h}$ under this isomorphism. We may find a common over field~$l_P$ of~$\kappa(P)$ that splits~$\alpha_{j, \kappa(P)}$ for all~$j$ such that~$v_l([l_P: \kappa(P)]) \leq \gssd_l^i(k)$.  Using a similar argument as in the previous paragraph, we may assume that~$l_P$ is a separable extension of~$\kappa(P)$. By \cite[Lemma 04GK]{stacks-project}, ~$l_P$ lifts to a finite {\'e}tale cover~$\Spec(Z_P) \to \Spec(R_P^h)$ of degree~$[l_P: \kappa(P)]$. Since~$R_P^h$ is a regular and local domain,~$Z_P$ must also be a domain. Let~$E_P$ be the fraction field of~$Z_P$. By pulling back along the generic fibre,~$E_P$ splits~$\alpha_j$ for all~$j$. By construction,~$[E_P: F_P^h] = [l_P: \kappa(P)]$. Hence,~$v_l([E_P: F_P^h]) = v_l([l_P: \kappa(P)]) \leq \gssd_l^i(k)$. Let~$d_2$ be the least common multiple of the degrees~$[E_P: F_P^h]$ for all $P\in \mathcal{P}$. By construction~$v_l(d_2) \leq \gssd_l^i(k)$.
    
    Now for each~$P \in \mathcal{P}$, by taking the direct sum of an appropriate number of copies of~$E_{P}$, we may obtain {\'e}tale~$F_{P}^h$-algebras~$L_{P}$ of a common degree~$d_2$. By Lemma~\ref{etale_small_patch}, there exists an {\'e}tale~$F$-algebra~$L_2$ of degree~$d_2$ such that~$L_2 \otimes_F F_{P}^h \cong L_{P}$ for all~$P \in \mathcal{P}$.

    Consider the \'etale~$F$-algebra~$L_1 \otimes_F L_2$ of degree~$d_1d_2$. It is the direct sum of copies of finite separable extensions of~$F$. 
    Since~$v_l(d_1d_2) = v_l(d_1) + v_l(d_2) \leq \gssd_l^i(k(x)) + \gssd_l^i(k)$, there exists a direct summand~$L$ of~$L_1 \otimes_F L_2$ such that~$v_l([L: F]) \leq \gssd_l^i(k(x)) + \gssd_l^i(k)$. 
    Let~$\mathcal{X}_L$ be the normalization of~$\mathcal{X}$ in~$L$, and~$X_L^{\red}$ be the reduced closed fibre of~$\mathcal{X}_L$. 
    Consider the natural map~$\mathcal{X}_L \rightarrow \mathcal{X}$. Let~$\mathcal{P}_L$ denote the preimage of the set~$\mathcal{P}$ under this map, and let~$\mathcal{U}_L$ denote the set of connected components of the complements of~$\mathcal{P}_L$ in~$X_L^{\red}$. 
    For~$P\in \mathcal{P}$,~$L \otimes_F F_P^h$ is the direct sum of the fields~$L_{\widetilde{P}}^h$, where~$\widetilde{P}$ ranges over the preimage of~$P$ under~$\mathcal{X}_L \rightarrow \mathcal{X}$ and~$L_{\widetilde{P}}^h$ denotes the Hensel patch constructed from~$\mathcal{X}_L$ at~$\widetilde{P}$. Similarly,~$L \otimes_F F_U^h$ is the direct sum of the fields~$L_{\widetilde{U}}^h$, where~$\widetilde{U}$ ranges over the preimage of~$U$ under~$\mathcal{X}_L \rightarrow \mathcal{X}$ and~$L_{\widetilde{U}}^h$ denotes the Hensel patch constructed from~$\mathcal{X}_L$ at~$\widetilde{U}$. By construction~$\alpha_j$ maps to the trivial class under the composition with the local-global map:
    \begin{equation*}
        \coho{F} \rightarrow \coho{L} \rightarrow \prod_{\xi\in \mathcal{P}_L \cup \mathcal{U}_L}\coho{L_{\xi}^h} \colon
        \alpha_j \mapsto \alpha_{j, L} \mapsto 0.
    \end{equation*} By Theorem~\ref{lgp_henselian},~$\alpha_{j, L}$ is trivial. Hence,~$L$ splits~$\alpha_j$ for all~$j$. 
    Hence, we obtain
    \begin{equation*}
        \operatorname{ind}(B)\mid l^{\gssd_l^i(k) + \gssd_l^i(k(x))}.
    \end{equation*}
    For part~$(a)$, take~$B = \{\alpha\}$. For each~$U \in \mathcal{U}$ and~$P \in \mathcal{P}$, we may argue the same as above but replace~$\gssd_l^i(k(x))$ and~$\gssd_l^i(k)$ by~$\ssd_l^i(k(x))$ and~$\ssd_l^i(k)$. 
\end{proof}

\subsection{A relative period-index bound for higher degree Galois cohomology classes}

Now we are ready to prove the main theorem of this section, which generalizes \cite[Thereom 2.9]{HHK23} to Hensel semi-global fields.

\begin{theorem}\label{main_theorem}
    Let~$K$ be an excellent henselian discretely valued field with residue field~$k$. Let~$F$ be a one-dimensional function field over~$K$. Let~$l$ be a prime that is not the characteristic of~$k$. Let~$k(x)$ denote the rational function field over~$k$. Assume~$i \geq 2$. Then
    \begin{equation*}
        \ssd_l^i(F) \leq \ssd_l^i(k) + \ssd_l^i(k(x)) + 
        \begin{cases}
            2 \textnormal{ if } l \neq 2\\
            3 \textnormal{ if } l = 2
        \end{cases}
    \end{equation*} and 
    \begin{equation*}
        \gssd_l^i(F) \leq \gssd_l^i(k) + \gssd_l^i(k(x)) + 
        \begin{cases}
            2 \textnormal{ if } l \neq 2\\
            3 \textnormal{ if } l = 2.
        \end{cases}
    \end{equation*}
\end{theorem}

\begin{proof}
    We first prove the statement for the generalized stable splitting dimension. Let~$B \subseteq \coho{F}$ be a collection of finitely many classes. By \cite[Proposition 3.1]{Gosavi22}, there exists a finte field extension~$L/F$ of degree~$l^2$ (resp.~$l^3$) for~$l \neq 2$ (resp.~$l = 2$) which corresponds to a morphism~$\mathcal{Y} \to \mathcal{X}$ for some regular model~$\mathcal{Y}$ of~$L$ such that~$\alpha_L \in H^i_{\textnormal{nr}}(L, \mu_l^{\otimes i-1})^{\mathcal{Y}}$. Here~$\mathcal{Y}$ is indeed a~$T$-curve as the construction of $L$ in the proof of \cite[Proposition 3.1]{Gosavi22} does not introduce algebraic elements over $K$.
    By Proposition~\ref{split_unramified} and also the definition of the general splitting dimension, there exists a finite extension~$L'/L$ that splits all~$\alpha\in B$ and~$v_l([L': L]) \leq \gssd_l^i(k(x)) + \gssd_l^i(k)$, so
    \begin{equation*}
        v_l([L': F]) \leq \gssd_l^i(k) + \gssd_l^i(k(x)) + 
        \begin{cases}
            2 \textnormal{ if } l \neq 2\\
            3 \textnormal{ if } l = 2.
        \end{cases}
    \end{equation*}
    
    We also have to consider cohomology classes in~$\coho{E}$ for all finite extensions ~$E/F$ to bound the generalized stable splitting dimension. Let~$B \subseteq \coho{E}$. Note that~$E$ is the function field of a curve defined over a finite algebraic extension of~$K$ whose residue field ~$k_E$ is also an algebraic extension over~$k$. By an argument as in the previous paragraph, we can show that there exists a finite field extension~$L/E$ such that~$L$ is a splitting field for~$B$ with the same residue field as that of~$E$, and 
    \begin{align*}
        v_l([L: E]) &\leq \gssd_l^i(k_E) + \gssd_l^i(k_E(x)) + 
        \begin{cases}
            2 \textnormal{ if } l \neq 2\\
            3 \textnormal{ if } l = 2
        \end{cases}\\
        &\leq \gssd_l^i(k) + \gssd_l^i(k(x)) + 
        \begin{cases}
            2 \textnormal{ if } l \neq 2\\
            3 \textnormal{ if } l = 2,
        \end{cases}
    \end{align*} whence the assertion for~$\gssd_l^i(F)$ follows.

    For the part concerning the stable splitting dimension, we may take~$B = \{\alpha\}$ and repeat the argument as above.
\end{proof}

\subsection{Iterative bounds for higher rank excellent henselian discretely valued fields} We have shown in the previous section that~$\gssd_l^i(F)$ is determined by the underlying geometry from the closed fibre of the model~$\mathcal{X}$ of~$F$. It is not hard to see that~$\gssd_l^i(F)$ is an iterative bound and therefore we may study such a bound for one-variable function fields~$F$ over ``higher rank'' excellent henselian discretely valued fields~$k_r$. Here~$k_r$ is constructed iteratively from a sequence of fields~$k_0, \dots, k_r$, where for each~$i \geq 1$,~$k_i$ is an excellent henselian discretely valued field with residue field~$k_{i-1}$. The results in this subsection mostly follows from \cite[Section 6]{HHK23} with necessary adaptions to henselian fields.


\begin{lemma}\label{iterative_lemma}
    Let~$k$ be a field and char$(k) \neq l$, where~$l$ is prime. Let the sequence of fields~$k_0\coloneqq k, k_1, \dots, k_r$ be constructed as above. Then for any finite collection of classes~$B \subseteq \coho{k_r}$, there exists an extension~$L/k_r$ of degree dividing~$l^{\gssd_l^i(k) + r}$ that splits~$B$. That is,~$\gssd_l^i(k_r)\leq \gssd_l^i(k) + r$. In particular, ~$\ssd_l^i(k_r)\leq \ssd_l^i(k) + r$.
\end{lemma}

\begin{proof} 
    By induction, we may reduce to the case of~$r = 1$. Let~$K \coloneqq k_1$ with valuation~$v$. Let~$\mathcal{O}_v$ denote the valuation ring of~$K$ with uniformizer~$\pi$. By \cite[Exp. XII, Corollaire 5.5]{SGA73}, the reduction map~$\coho{\mathcal{O}_v} \rightarrow \coho{k}$ is an isomorphism of groups, for~$i \geq 1$, so we identify these two groups from now on. By \cite[Proposition 1.4.6]{CTS21} (see also \cite[7.9]{GMS03}), for any~$\alpha \in \coho{K}$,~$\alpha = \alpha' + (\pi) \cup \beta$, where~$\alpha' \in \coho{\mathcal{O}_v}$ and~$\beta \in H^{i-1}(k, \mu_l^{\otimes i-2})$. Since~$\widetilde{K}\coloneqq K(\pi^{1/l})$ splits~$(\pi)$,~$(\alpha)_{\widetilde{K}} \sim (\alpha')_{\widetilde{K}} \in \coho{\widetilde{K}}$ by functoriality.

    Now consider a finite collection~$B = \{\alpha_1, \dots, \alpha_m\}\subseteq \coho{K}$ and the corresponding~$B' = \{\alpha_1', \dots, \alpha_m'\}$, where each~$\alpha_j'$ is obtained from~$\alpha_j$ as shown above. By definition, there exists a field extension~$k'/k$ that splits~$B'$ such that~$v_l([k': k]) \leq \gssd_l^i(k)$. Therefore, it suffices to show that there exists a splitting field~$\widetilde{K}'/K$ such that~$[\widetilde{K}': K] \mid l[k': k]$. By the assumption on char$(k)$, any class in~$B$ splits over a separable closure of~$k$ in~$k'$. Therefore, we may assume~$k'/k$ is separable. Hence, we may construct an unramified extension~$\mathcal{O}_v'$ of~$\mathcal{O}_v$ with residue field~$k'$ and the degree of extension is exactly~$[k': k]$. Let~$K'$ denote the fraction field of~$\mathcal{O}_v'$. By \cite[Exp. XII, Corollaire 5.5]{SGA73} again,~$(\alpha_j')_{\mathcal{O}_v'}$ splits, then so does~$(\alpha_j')_{K'}$. Now let~$\widetilde{K}'\coloneqq \widetilde{K}K'$. Then~$(\alpha_j')_{\widetilde{K}'} = (\alpha_j)_{\widetilde{K}'} = 0$. Moreover, by construction,~$[\widetilde{K}': K] \mid l[k': k]$. Then we have shown that~$\gssd_l^i(k_1) \leq \gssd_l^i(k) + 1$. If~$B$ consists only one class, we may replace~$\gssd_l^i(k)$ by~$\ssd_l^i(k)$ and repeat the same argument as above.
\end{proof}

Combining this lemma with Theorem~\ref{main_theorem}, we get the following iterative bound for rational function fields over a higher rank excellent henselian discretely valued field.

\begin{theorem}\label{iterative_theorem}
    Let notation be as in Lemma~\ref{iterative_lemma}. Then
    \begin{equation*}
        \gssd_l^i(k_r(x)) \leq \gssd_l^i(k(x)) + r\gssd_l^i(k) + 
        \begin{cases}
            \frac{r}{2}(r + 3) \textnormal{ if } l \neq 2\\
            \frac{r}{2}(r + 5) \textnormal{ if } l = 2.
        \end{cases}
    \end{equation*} The same holds if we replace~$\gssd_l^i(-)$ by~$\ssd_l^i(-)$.
\end{theorem}

\begin{proof}
    By Theorem~\ref{main_theorem},~$\gssd_l^i(k_{j}(x)) \leq \gssd_l^i(k_j) + \gssd_l^i(k_{j-1}(x)) + \delta$, where~$\delta = 2$ if~$l \neq 2$ and~$\delta = 3$ if~$l = 2$. Then by \Cref{iterative_lemma},
    \begin{equation*}
        \gssd_l^i(k_j(x)) - \gssd_l^i(k_{j-1}(x)) \leq \gssd_l^i(k_j)  + \delta \leq \gssd_l^i(k) + (j - 1) + \delta.
    \end{equation*} By taking the sum of both sides of the inequalities from~$j = 1$ to~$j = r$ and an easy calculation, we obtain the desired conclusion. The same argument applies if we replace~$\gssd_l^i(-)$ by~$\ssd_l^i(-)$. We refer the reader to \cite[Theorem 6.3]{HHK23} for details.
\end{proof}

Now we study the bound for the generalized splitting dimension as the cohomological degree varies, taking into account the cohomological dimension of~$k$. In general, we can
only make statements for the case that the cohomological dimension is finite, leveraging the fact that~$\gssd^m_l(k) = 0$ for~$m > \operatorname{cd}_l(k)$.

\begin{theorem}
    Let notation be as in Lemma~\ref{iterative_lemma}. Let~$c \coloneqq \cd_l(k)$. Then 
    \begin{equation*}
        \gssd_l^{c + m}(k_r) \leq \max(0, r - m + 1) \textnormal{ for } m \geq 1,
    \end{equation*} and 
    \begin{equation*}
        \gssd_l^{c + m}(F) \leq
        \begin{cases}
            \frac{1}{2}r(r - 1) + r\delta + \gssd_l^{c + 1}(k(x)) & \textnormal{ for } m = 1,\\
            \frac{1}{2}(r - m + 1)(r - m) + (r - m + 2)\delta & \textnormal{ for } 2 \leq m \leq r + 1,\\
            0 & \textnormal{ for } m > r + 1,
        \end{cases}
    \end{equation*} for any one-variable function field~$F$ over~$k_r$, where~$\delta = 2$ if~$l$ is odd and~$\delta = 3$ if~$l = 2$. The same holds if we replace~$\gssd_l^i(-)$ by~$\ssd_l^i(-)$.
\end{theorem}

\begin{proof} Let~$\widehat{k}_r$ denote the completion of~$k_r$. By \cite[Proposition 3.5.3]{GGMB14},~$H^i(k_r, G) \to H^i(\widehat{k}_r, G)$ is bijective for any commutative group~$G$. Therefore,~$\operatorname{cd}_l(k_r) = \operatorname{cd}_l(\widehat{k}_r)$, for any prime~$l$. Hence, by \cite[Proposition II.4.3.12]{serre97}, we may inductively show that~$\cd_l(k_r) = \cd_l(k) + r = c + r$. Therefore, for~$m > r$,~$\gssd_l^{c + m}(k_r) = 0$. For~$m \leq r$,~$\gssd_l^{c + m}(k_{j}) = 0$ for~$j \leq m - 1$. Similarly, by \cite[Proposition II.4.2.11]{serre97},~$\cd_l(k_r(x)) = c + r + 1$. The rest of the calculation follows from \cite[Theorem 6.4]{HHK23}.
\end{proof}

\section{The Period-index bound for coefficient groups of non-prime order}\label{sec_non_prime}
In the previous section, we studied the period-index bound for the higher degree Galois cohomology group~$\coho{F}$, where~$l$ is prime. Now we study the same problem for coefficient groups of non-prime orders.
Throughout this section, we assume that char$(F) \nmid m$, where~$m$ is a positive integer.

\subsection{The period-index problem for $\cohom{F}{m}$}

\begin{definition}\label{period}
  Let~$\alpha$ be a non-trivial element in~$\cohom{F}{m}$. The \emph{period} of~$\alpha$, denoted by~$\per(\alpha)$, is the smallest integer~$n$ such that~$n\alpha \sim 0 \in \cohom{F}{m}$. 
\end{definition}

When~$m = l$ is a prime number, the period of a non-trivial class in~$\coho{F}$ is always~$l$. The period-index problem is well-defined for prime coefficients since~$\ind(\alpha)$ always divides a power of~$l$ (see \cite[Lemma 2.3]{HHK23}). Now we show that the period-index problem is also well-defined when~$m$ is not necessarily prime. We start by proving the case that~$m$ is a power of a prime number.

\begin{lemma}\label{index}
  Let~$F$ be a field and assume char$(F)$ is not divisible by the prime number~$p$. Let~$m = p^n$ for some positive integer~$n$. Let~$\alpha \in H^i(F, \mu_m^{\otimes i - 1})$. Then there exists a field extension~$L$ over~$F$ with degree a power of~$p$ that splits~$\alpha$. In particular, the index of~$\alpha$ is a power of~$p$ and therefore~$\ind(\alpha) \mid \per(\alpha)^s$ for some integer~$s$.
\end{lemma}

\begin{proof}
  We prove the lemma by inducting on the exponent~$n$ of~$p$. By \cite[Lemma 2.3]{HHK23}, the statement is true for~$n = 1$. Suppose the statement holds for~$n = k$. That is, there exists a splitting field of degree~$\operatorname{ind}(\alpha)$, which is a power of~$p$, for all~$\alpha \in H^i(F, \mu_{p^k}^{\otimes i - 1})$.

  For~$n = k + 1$, consider the short exact sequence of Galois cohomology groups:
  \begin{equation*}
    \cdots \longrightarrow H^{i}(F, \mu_p^{\otimes i - 1}) \overset{\iota}{\longrightarrow} H^i(F, \mu_{p^{k + 1}}^{\otimes i - 1}) \overset{\pi}{\longrightarrow} H^i(F, \mu_{p^k}^{\otimes i - 1})\longrightarrow \cdots
  \end{equation*}  induced by the exact sequence of Galois modules $0 \longrightarrow\mu_p \longrightarrow \mu_{p^{k + 1}} \longrightarrow \mu_{p^k} \longrightarrow 0 $.
  
  Let~$\alpha$ be a non-trivial element in~$H^i(F, \mu_{p^{k + 1}}^{\otimes i - 1})$ and~$\beta\coloneqq \pi(\alpha) \in H^i(F, \mu_{p^k}^{\otimes i - 1})$. By the induction hypothesis, there exists a splitting field~$L/F$ of~$\beta$ such that~$[L: F] = \operatorname{ind}(\beta) = p^a$ for some non-negative integer~$a$ (take~$a = 0$ if~$\beta \sim 0$). Then~$\pi_L(\alpha_L) = 0$. 
  By the exactness, there exists a preimage~$\gamma \in H^i(L, \mu_p^{\otimes i-1})$ of~$\alpha_L$ under~$\iota$. Now by the base case, there exists a splitting field~$K/L$ of~$\gamma$ such that~$[K: L] = \operatorname{ind}(\gamma) = p^b$ for some non-negative integer~$b$ (take~$b = 0$ if~$\alpha_L$ is already trivial). Then~$\alpha_K = \iota_K(\gamma_K) \sim \iota_K(0) = 0$. Hence, $\ind(\alpha) \mid [K: F] = p^{a+b}$.
\end{proof}

The following two lemmas are known. We give the proof here for the convenience of the reader.

\begin{lemma}\label{composite_index}
  Let~$m~$ be a positive integer. Let~$\alpha \in H^i(F, \mu_m^{\otimes i - 1})$. Let~$L$ be a finite field extension over~$F$ of degree~$e$ prime to~$\ind(\alpha)$.  Then~$\operatorname{ind}(\alpha) = \operatorname{ind}(\alpha_L)$.
\end{lemma}

\begin{proof}
  Let~$F'/F$ be a splitting field of~$\alpha$ of degree~$\operatorname{ind}(\alpha)$. Let~$L'$ be the composite of~$F'$ and~$L$. Hence,~$\ind(\alpha_L) \mid [L': L] = \ind(\alpha)$ since~$\operatorname{ind}(\alpha)$ and~$[L: F]$ are coprime. Moreover, let~$K/L$ be a splitting field of~$\alpha_L$ of degree~$\operatorname{ind}(\alpha_L)$. Then~$K/F$ also splits~$\alpha$, so~$\operatorname{ind}(\alpha) \mid [K: F] = e\operatorname{ind}(\alpha_L)$. Since~$\operatorname{ind}(\alpha)$ is prime to~$e$, we must have~$\operatorname{ind}(\alpha) = \operatorname{ind}(\alpha_L)$.
\end{proof}

\begin{lemma}\label{res_cor}
  Let~$m$ be a positive integer and let~$L$ be a finite field extension over~$F$ of degree prime to~$\per(\alpha)$. Let~$\alpha \in H^i(F, \mu_m^{\otimes i - 1})$. Then~$\per(\alpha) = \per(\alpha_L)$. 
\end{lemma}

\begin{proof} 
  
  Consider the composition of the restriction and the corestriction map:
  \begin{equation*}
    \varphi \colon H^i(F, \mu_m^{\otimes i - 1})  \overset{\textnormal{Res}}{\longrightarrow} H^i(L, \mu_m^{\otimes i - 1}) \overset{\textnormal{Cor}}{\longrightarrow} H^i(F, \mu_m^{\otimes i - 1})
  \end{equation*} 
  By \cite[Proposition 4.2.12]{GS17},~$\varphi(\alpha) = e\alpha$, where~$e \coloneqq [L: F]$ is coprime to~$m$. 

  Let~$m_1 \coloneqq \per(\alpha)$ and let~$m_2 \coloneqq \per(\alpha_L)$. First, since the map~$ \textnormal{Res}\colon m_1\alpha \sim 0 \mapsto m_1\alpha_L \sim 0 \in H^i(L, \mu_m^{\otimes i - 1})$,~$m_2$ must divide~$m_1$.  The fact that~$\varphi(m_2\alpha) = m_2 e\alpha = 0$ implies that~$m_1\mid m_2e$. Since~$e$ is coprime to~$m$, we conclude that~$m_1 \mid m_2$. Hence,~$\per(\alpha) = \per(\alpha_L)$.
\end{proof}

Now we show that the period-index problem is well-defined for all classes~$\alpha \in H^i(F, \mu_m^{\otimes i - 1})$ for any positive integer~$m$. 

\begin{proposition}\label{non-prime}
  Let~$m$ be a positive integer. Let~$F$ be a field and char$(F) \nmid m$. Suppose~$m = p_1^{e_1}p_2^{e_2}\cdots p_k^{e_k}$, where each~$p_i$ is prime and pairwisely distinct. Let~$\alpha \in H^i(F, \mu_m^{\otimes i - 1})$.
  Then~$\per(\alpha)$ and~$\operatorname{ind}(\alpha)$ have the same prime factors as~$m$. 
  Moreover,~$\per(\alpha) \mid \operatorname{ind}(\alpha) \mid \per(\alpha)^c$ for some positive integer~$c$.
\end{proposition}
\begin{proof}
  Since~$\mu_m$ is abelian,~$\mu_m \cong \mu_{p_1^{e_1}} \times \cdots \times \mu_{p_n^{e_n}}$. Therefore, we obtain the following isomorphism of Galois cohomology groups:
  \begin{align*} 
    h: H^i(F, \mu_m^{\otimes i - 1}) &\overset{\sim}{\longrightarrow} H^i(F, \mu_{p_1^{e_1}}^{\otimes i - 1}) \times \cdots \times H^i(F, \mu_{p_k^{e_k}}^{\otimes i - 1}),\\
    \alpha &\longmapsto (\alpha_1, \cdots, \alpha_k)
  \end{align*}where~$\alpha_s \in H^i(F, \mu_{p_s^{e_s}}^{\otimes i - 1})$. 
  Thus, it suffices to study the splitting behavior for each~$\alpha_s$ with respect to each prime factor~$p_s$. 
  Let~$F_1/F$ be a splitting field of~$\alpha_1$ such that~$[F_1: F] = \operatorname{ind}(\alpha_1)$. Note that such an~$F_1$ exists by Lemma~\ref{index}. Then~$h(\alpha_{F_1}) = (0, \alpha_{2, F_1}, \cdots, \alpha_{k, F_1})$.
  We may repeat this process and inductively construct splitting fields~$F_{s+1}/F_s$ for ~$\alpha_{s+1, F_s}$ with degree~$[F_{s+1}: F_s] = \operatorname{ind}(\alpha_{s+1, F_s}) = \operatorname{ind}(\alpha_{s+1})$  (see Lemma~\ref{composite_index}) for each~$s$. Therefore, we may finally construct~$F_k/F$ that splits~$\alpha$ and~$[F_k: F] = \prod_{s = 1}^k \operatorname{ind}(\alpha_{s}) = p_1^{a_1} \cdots p_k^{a_k}$ for non-negative integers~$a_s$ by Lemma~\ref{index} ($a_s = 0$ if~$\alpha_s \sim 0$).

  Let~$\per(\alpha) = p_1^{b_1} \cdots p_k^{b_k}$, where~$b_s \leq e_s$ for each~$s$. Then~$\per(\alpha_s) = p_s^{b_s}$ for each component~$\alpha_s \in h(\alpha)$. By Lemma~\ref{res_cor},~$\per(\alpha_{s, F_{s-1}})$ is also~$p_s^{b_s}$, and hence~$\ind(\alpha_{s, F_{s-1}}) = [F_s: F_{s-1}] \mid \per(\alpha_s)^{c_s}$ for some integer~$c_s$. Therefore, 
  \begin{equation*}
    \ind(\alpha) \mid \prod_{s = 1}^k \ind(\alpha_{s, F_{s-1}}) \mid \prod_{s=1}^k \per(\alpha_s)^{c_s} \mid \per(\alpha)^c,
  \end{equation*} where~$c$ can be taken to be any integers greater than or equal to the maximum of~$c_1, \dots, c_k$.
\end{proof}

\subsection{A period-index bound for~$\cohom{F}{m}$}

As shown in the previous section, the period-index problem is well-defined for the Galois cohomology group~$H^i(F, \mu_m^{\otimes i-1})$. Accordingly, we may make the following definition. 

\begin{definition}
  Fix a degree~$i$. The \emph{i-splitting dimension at m of F} is the minimal exponent~$n$ of~$\per(\alpha)$, denoted by~$\operatorname{sd}_m^i(F)$, so that for all~$\alpha \in \cohom{F}{m}$, ~$\ind(\alpha)$ divides~$\per(\alpha)^n$. The \emph{stable i-splitting dimension at m of F}, denoted by ~$\operatorname{ssd}_m^i(F)$, is the minimal exponent~$n$ so that~$\operatorname{sd}_m^i(L) \leq n$ for all finite field extensions~$L/F$. 
\end{definition}

We will show that there also exists a bound for~$\operatorname{ssd}_m^i(F)$, when~$F$ is a Hensel semi-global field and char$(F)$ does not divide~$m$. We first study such a bound for~$m = l^n$, where~$l$ is a prime number and char$(F)$ is not~$l$.

\begin{lemma}\label{K-group}
  For any~$\alpha \in H^i(F, \mu_{l^n}^{\otimes i -1})$ with~$\per(\alpha) = l^s$, there exists~$\alpha' \in H^i(F, \mu_{l^s}^{\otimes i-1})$ that is in the preimage of~$\alpha$ under the natural map 
  \begin{equation*}
    \cdots \longrightarrow H^i(F, \mu_{l^s}^{\otimes i-1}) \longrightarrow H^i(F, \mu_{l^n}^{\otimes i-1}) \longrightarrow H^i(F, \mu_{l^{n-s}}^{\otimes i-1}) \longrightarrow \cdots
  \end{equation*} 
  induced by the short exact sequence
  \begin{equation*}
    1 \longrightarrow \mu_{l^s} \longrightarrow \mu_{l^n} \longrightarrow \mu_{l^{n-s}} \longrightarrow 1
  \end{equation*}
\end{lemma}

\begin{proof}
  Consider the following short exact sequences of groups:
\begin{center}
  \begin{tikzcd}
      & &1 \arrow{d}{}& & &\\
      & &\mu_{l^{s}}\arrow{d}{} & & &\\
      & &\mu_{l^{n}}\arrow{dr}{s}\arrow{d}{\pi_1} & & &\\
     &1 \arrow{r}{} &\mu_{l^{n-s}} \arrow{d}{}\arrow{r}{\iota}& \mu_{l^n} \arrow{r}{\pi_2} & \mu_{l^{s}} \arrow{r}{}& 1\\
     & &1 & &  &
  \end{tikzcd}
\end{center} Here the diagonal map is given by~$s \colon x \mapsto x^{l^s}$. This induces the following diagram on~$i$-th degree Galois cohomology groups:
\begin{center}
  \begin{tikzcd}
     & & &\vdots \arrow{d}{}& & \\
     & & &H^i(F, \mu_{l^{s}}^{\otimes i-1})\arrow{d}{} & & \\
     & & &H^i(F, \mu_{l^{n}}^{\otimes i-1})\arrow{dr}{H^i(s)}\arrow{d}{H^i(\pi_1)} & & \\
     &H^{i - 1}(F, \mu_{l^{n}}^{\otimes i - 1})\arrow{r}{H^{i-1}(\pi_2)} &H^{i - 1}(F, \mu_{l^{s}}^{\otimes i - 1}) \arrow{r}{\partial}  &H^i(F, \mu_{l^{n-s}}^{\otimes i-1}) \arrow{d}{}\arrow{r}{H^i(\iota)}& H^i(F, \mu_{l^{n}}^{\otimes i-1}) \arrow{r}{H^i(\pi_2)} & \cdots \\
     & & &\vdots & &  
  \end{tikzcd}
\end{center} To prove the statement, it suffices to show that~$H^i(\iota)$ is injective. Indeed, since~$\per(\alpha) = l^s$, we must have~$H^i(s)\colon \alpha \mapsto 0$, and if~$H^i(\iota)$ is injective,~$H^i(\pi_1)\colon \alpha \mapsto 0$. By exactness, this implies that there exists a preimage~$\alpha' \in H^i(F, \mu_{l^{s}}^{\otimes i-1})$ that maps to~$\alpha$. This proves the assertion.

Therefore, it suffices to show that~$\ker(H^i(\iota)) = \coker(H^{i - 1}(\pi_2))$ is trivial. That is, we show that~$H^{i - 1}(F, \mu_{l^{n}}^{\otimes i - 1}) \rightarrow H^{i - 1}(F, \mu_{l^{s}}^{\otimes i - 1})$ is surjective. By the Norm-residue isomorphism theorem (\cite{Voevodsky}),~$H^{i - 1}(F, \mu_{l^s}^{\otimes i - 1}) \cong K_{i - 1}^{M}(F)/(l^s)$, where~$K_{i-1}^M$ denotes the Milnor-$K$-group. Then for~$\beta \in H^{i - 1}(F, \mu_{l^s}^{\otimes i - 1})$, we may write~$\beta = \sum_{j = 1}^n (\beta_{j1}) \cup \cdots \cup (\beta_{j(i-1)})$, where~$\beta_{jk} \in H^1(F, \mu_{l^s})$. It suffices to show that~$H^1(F, \mu_{l^n}) \rightarrow H^1(F, \mu_{l^s})$ is surjective. By Kummer theory (see for example, \cite[Proposition 4.3.6]{GS17}), this is equivalent to showing that~$F^{\times}/(F^{\times})^{l^n} \rightarrow F^{\times}/(F^{\times})^{l^s}$ is surjective. This is obvious.  
\end{proof}

\begin{proposition}\label{period-index_non_prime} Let~$F$ be a field,~$l$ a prime that is not the characteristic of~$F$. Then~$\textnormal{ssd}_{l^n}^i(F) \leq \textnormal{ssd}_{l}^i(F)$. 
\end{proposition}

\begin{proof}
  We prove the statement by induction, using a similar argument as in the proof of Lemma~\ref{index}. We have shown the statement for~$l = 1$. Suppose the statement holds for all~$n \leq k$. Then for~$n = k + 1$, take~$\alpha \in H^i(F, \mu_{l^{k + 1}}^{\otimes i - 1})$. Suppose~$\per(\alpha) = l^s$ for some~$s < k + 1$, then by Lemma~\ref{K-group}, in the following exact sequence
  \begin{equation*}
    \cdots \longrightarrow H^i(F, \mu_{l^s}^{\otimes i - 1}) \overset{i}{\longrightarrow} H^i(F, \mu_{l^{k+1}}^{\otimes i-1}) \overset{\pi}{\longrightarrow} H^i(F, \mu_{l^{k+1-s}}^{\otimes i-1}) \longrightarrow \cdots,
  \end{equation*} there exists~$\alpha' \in H^i(F, \mu_{l^s}^{\otimes i-1})$ such that~$\alpha' \mapsto \alpha$. Note that~$\per(\alpha') = l^s$ (otherwise~$\per(\alpha)$ must be smaller). By the induction hypothesis,~$\alpha'$ is split by a finite field extension over~$F$ with degree at most~$(l^s)^{\textnormal{ssd}_l^i(F)}$. Then so is~$\alpha$.

  Now suppose~$\per(\alpha) = l^{k+1}$. Let~$\beta\coloneqq \pi(\alpha) \in H^i(F, \mu_{l^{k+1 - s}}^{\otimes i - 1})$. Note that~$\per(\beta) = l^{k+1-s}$. Indeed, if~$\per(\beta) = l^{k+1-s-j}$, for some~$j > 0$, there must exists a preimage~$\gamma$ of~$l^{k+1-s-j}\alpha$ under the natural map~$i$ by Lemma~\ref{K-group}. However,~$\per(\gamma) \leq l^s$, then so is~$\per(l^{k+1-s-j}\alpha)$. Therefore,~$j = 0$. Then there exists a finite field extension~$L/F$ of degree less that~$(l^{k+1-s})^{\textnormal{ssd}_l^i(F)}$ that splits~$\beta$,~$\beta_L = 0$. Then there exists a preimage~$\alpha'' \in H^i(L, \mu_{l^s}^{\otimes i-1})$ of~$\alpha_L$ under~$i$. Note that~$\per(\alpha'') \mid l^s$. Similarly,~$\alpha''$ is split by a finite field extension~$K$ over~$L$ with degree at most~$(l^s)^{\textnormal{ssd}_l^i(L)}$. Then so is~$\alpha_L$. 
  Hence,~$\alpha$ is split by~$K$, and 
  \begin{equation*}
    [K: F] \leq (l^s)^{\textnormal{ssd}_l^i(L)} \cdot (l^{k+1-s})^{\textnormal{ssd}_l^i(F)} \leq (l^s)^{\textnormal{ssd}_l^i(F)} \cdot (l^{k+1-s})^{\textnormal{ssd}_l^i(F)} = (l^{k+1})^{\textnormal{ssd}_l^i(F)}.
  \end{equation*} We may also repeat the above process for any finite field extension~$E/F$. Therefore, we have shown that~$\textnormal{ssd}_{l^n}^i(F) \leq \textnormal{ssd}_{l}^i(F)$.
\end{proof}

\begin{theorem}\label{general_theorem}
  Let~$m = l_1^{e_1}\cdots l_k^{e_k}$, where each~$l_s$ is a prime number and pairwisely distinct. Assume~$\operatorname{char}(F) \nmid m$. Then~$\operatorname{ssd}_m^i(F) \leq \max_{1 \leq s \leq k} \{\operatorname{ssd}_{l_s}^i(F)\}$.
\end{theorem}

\begin{proof}
  We may argue as in Proposition~\ref{non-prime}. Since~$\mu_m$ is abelian,~$\mu_m \cong \mu_{l_1^{e_1}} \times \cdots \times \mu_{l_k^{e_k}}$. Therefore, 
  \begin{align*}
    h: H^i(F, \mu_m^{\otimes i - 1}) &\overset{\sim}{\longrightarrow} H^i(F, \mu_{l_1^{e_1}}^{\otimes i - 1}) \times \cdots \times H^i(F, \mu_{l_k^{e_k}}^{\otimes i - 1}),\\
    \alpha &\longmapsto (\alpha_1, \cdots, \alpha_k)
  \end{align*}where~$\alpha_s \in H^i(F, \mu_{l_s^{e_s}}^{\otimes i - 1})$. 
  
  First, by Proposition~\ref{period-index_non_prime}, we may construct~$F_1/F$ of degree~$\per(\alpha_1)^{\textnormal{ssd}_{l_1}(F)}$ that splits~$\alpha_1$. Therefore ~$h(\alpha_{F_1}) = (0, \alpha_{2, F_1}, \cdots, \alpha_{k, F_1})$. We may repeat this process and inductively construct a finite field extension~$F_{s+1}/F_s$ that splits~$\alpha_{s+1, F_s}$ with degree at most~$[F_{s+1}: F_s] = \per(\alpha_{s+1, F_s})^{\textnormal{ssd}_{l_s}^i(F_s)}$. Also, by Lemma~\ref{res_cor}, we see that~$\per(\alpha_{s+1, F_s}) = \per(\alpha_{s+1})$ since~$[F_s: F]$ is coprime to~$l_{s+1}$. Therefore, we may finally construct~$F_k/F$ that splits~$\alpha$ and 
  \begin{align*}
    [F_k: F] &= \prod_{s = 1}^k \per(\alpha_{s})^{\textnormal{ssd}_{l_s}^i(F_{s-1})} \leq \prod_{s = 1}^k \per(\alpha_{s})^{\textnormal{ssd}_{l_s}^i(F)} \leq \prod_{i = 1}^k \per(\alpha_{s})^{\max_{1 \leq s \leq k} \{\operatorname{ssd}_{l_s}^i(F)\}} 
  \end{align*} Since~$\per(\alpha) = \per(h(\alpha))$ and~$\per(\alpha_j) \mid l_s^{e_s}$, we must have~$\per(\alpha) = \per(\alpha_1) \cdots \per(\alpha_k)$. Therefore,~$[F_k: F] \leq \per(\alpha)^{\max_{1 \leq s \leq k} \{\operatorname{ssd}_{l_s}^i(F)\}}$, whence~$\operatorname{sd}_m^i(F) \leq \max_{1 \leq s \leq k} \{\operatorname{ssd}_{l_s}^i(F)\}$.

  To prove the assertion for the stable splitting dimension, we also need to consider all finite field extensions~$E/F$. We may repeat the argument above and show that~$\operatorname{sd}_m^i(E) \leq \ssd_{l_s}^i(E) \leq \ssd_{l_s}^i(F)$ by \cite[Proposition 2.7]{HHK23}.
\end{proof}

\begin{remark}
    Note that Theorem~\ref{general_theorem} is a relative bound, only depending on the bound for the stable splitting dimension at the prime factors of~$m$. This bound holds for arbitrary fields.
\end{remark}

\subsection{Example: Complex threefolds} In \cite{deJong}, de Jong proved that~$\per(\alpha)$ equals~$\ind(\alpha)$ for any~$\alpha$ in the Brauer group of the function field of an algebraic surface over an algebraically closed field.
In this subsection, we give an example of the period-index bound of~$H^3(F, \mu_m^{\otimes 2})$ for certain 3-dimensional fields~$F$ over algebraically closed fields.  
\begin{corollary}
    Let~$k$ be the function field of an algebraic curve over an algebraically closed field and let~$K$ be the field of algebraic Laurent series over~$k$. Let~$F$ be a one-variable function field over~$K$. Then 
    \begin{equation*}
        \ssd_m^3(F) \leq 
        \begin{cases}
            2, \quad m \textnormal{ is odd},\\
            3, \quad m \textnormal{ is even}.
        \end{cases}
    \end{equation*} If we further assume that~$m$ is prime, the same assertion also holds for the generalized stable splitting dimension. 
\end{corollary}

\begin{proof}
    Since the cohomological dimension of~$K$ and~$F$ are 1 and 2 respectively,~$\gssd_l^3(K) = \gssd_l^3(F) = 0$ for any prime number~$l$. Therefore, by Theorem~\ref{iterative_theorem},~$\gssd_l^3(F) \leq 2$ if~$l \neq 2$ and~$\gssd_2^3(F) = 3$. By a similar argument, we may also show the same assertion for~$\ssd_l^3(F)$.

    By Theorem~\ref{general_theorem}, 
    \begin{equation*}
        \ssd_m^3(F) \leq \max\{\ssd_{l_s}^3(F)\} \leq
        \begin{cases}
            2, \quad 2 \nmid m,\\
            3, \quad 2 \mid m.
        \end{cases}
    \end{equation*}
\end{proof}

\bibliographystyle{alpha}
\newcommand{\etalchar}[1]{$^{#1}$}

\noindent \textbf{Author information}:

\medskip

\noindent Yidi Wang: Department of Mathematics, University of Western Ontario, London, Ontario, N6H 0B4, Canada.\\ Email: {\tt ywan6443@uwo.ca}

\end{document}